\documentclass{aims}
\usepackage{amsmath,amssymb, amstext, amsthm,amsfonts,euscript,amscd,mathrsfs}
  \usepackage{paralist}
  \usepackage{graphics} 
  \usepackage{epsfig} 

 \usepackage[colorlinks=true]{hyperref}
\hypersetup{urlcolor=blue, citecolor=red}

\def\currentvolume{X}
 \def\currentissue{X}
  \def\currentyear{200X}
   \def\currentmonth{XX}
    \def\ppages{X--XX}
\newcommand{\ep}{\varepsilon}

\DeclareMathOperator{\R}{\mathbb{R}}

\DeclareMathOperator{\Z}{\mathbb{Z}}
\newcommand{\Rn}{\R^{n}}

\newcommand{\Schr}{Schr\"{o}dinger }
\newcommand{\jap}[1]{\!\left<#1\right>}
\providecommand{\norm}[1]{\lVert#1\rVert}
\providecommand{\abs}[1]{\lvert#1\rvert}

 \newcommand{\linf}{L^{\infty}}
 \newcommand{\les}{\lesssim}
 \newcommand{\BM}{\left[-M,M\right]}
\newcommand{\dx}{\partial_x}
\newcommand{\dxi}{\partial_{\xi}}
\newcommand{\laplace}{\triangle}
\newcommand{\grad}{\nabla}
\newcommand{\der}[2]{\frac{\partial #1}{\partial #2}}
\newcommand{\dt}{\partial_t}
\newcommand{\weakto}{\rightharpoonup}
\newcommand{\tld}[1]{\tilde{#1}}
\newcommand{\vu}{\vec{u}}
  \newcommand{\vue}{\vu^\ep}
    \newcommand{\vuee}{\vu^{\ep'}}
 \newcommand{\vv}{\vec{v}}

\newcommand{\vf}{\vec{f}}
\newcommand{\vg}{\vec{g}}
\newcommand{\vz}{\vec{z}}
\newcommand{\dz}{\partial_{\vz}}
  \newcommand{\vo}{\vec{0}}
\newcommand{\labelen}[1]{\renewcommand{\labelenumi}{(#1\arabic{enumi})} }
\newcommand{\Schw}{\mathscr{S}}
\newcommand{\B}{\mathcal{B}}
  \newcommand{\K}{\mathcal{K}}
\newcommand{\tb}{\tld{b}}
\newcommand{\tc}{\tld{c}}
\newcommand{\td}{\tld{d}}

\newtheorem{theorem}{Theorem}[section]
\newtheorem{corollary}{Corollary}

\newtheorem{lemma}[theorem]{Lemma}
\newtheorem{proposition}{Proposition}

\theoremstyle{definition}

\newtheorem{remark}{Remark}

\numberwithin{equation}{section}

\title[wellposedness of quasi-linear KdV]{Local well posedness of quasi-linear systems generalizing KdV}
\author{Timur-Akhunov}


  \email{takhunov@ucalgary.ca}
  
\begin{document}
\maketitle
     \centerline{\scshape Timur Akhunov }
\medskip
{\footnotesize
 \centerline{Department of Mathematics, University of Calgary}
 \centerline{2500 University Drive NW}
   \centerline{ Calgary, Alberta, T2N 1N4, Canada}
} 

\bigskip

 \centerline{(Communicated by the associate editor name)}

\begin{abstract}
In this article we prove local well-posedness of quasilinear dispersive systems of PDE generalizing KdV. These results adapt the ideas of Kenig-Ponce-Vega from the Quasi-Linear {\Schr}equations to the third order dispersive problems. The main ingredient of the proof is a local smoothing estimate for a general linear problem that allows us to proceed via the artificial viscosity method.
\end{abstract}
\section{Introduction}
In this paper we consider the following system of PDE:
\begin{equation}\label{eq}
  \begin{cases}
    & \dt \vu + a(x,t,\vu,\dx\vu,\dx^2\vu)\cdot\dx^3\vu + b(x,t,\vu,\dx\vu,\dx^2\vu)
    \cdot\dx^2\vu\\& + c(x,t,\vu,\dx\vu)\cdot\dx\vu + d(x,t,\vu)\cdot\vu = \vf(x,t) \text{ on } \R\times I\\
    & \vu(x,0) = \vu_0(x) \text{ on } \R\times\{0\}
  \end{cases}
\end{equation}
 where $I \subset [-1,1]$ is a time interval containing $0$; all functions are real valued; $\vu = (u_1,\ldots,u_n)$ is an unknown; $\vu_0=(u_{0,1},\ldots,u_{0,n})$ and $\vf = (f_1,\ldots,f_n)$ are given data; and coefficients $a$, $b$, $c$, $d$ are $n\times n$ matrix valued functions.\\

This way \eqref{eq} is a generalized KdV with the dispersive relation $a$ that depends on the unknown and its derivatives $(\vu,\dx \vu ,\dx^2\vu)$. \\

Moreover, well-posedness of \eqref{eq} allows one to study fully non-linear dispersive systems of the form:
     \begin{align*}
   \begin{cases}
     \dt \vv + f(x,t,\vv,\ldots \dx^3 \vv) = \vo \text{ for } (x,t) \in \R\times[-T,T]\\
     \vv(x,0) = \vv_0(x)
   \end{cases}
   \end{align*}
    if the non-linear function $\vf$ satisfies appropriate assumptions similar to the one for \eqref{eq}. Indeed, differentiating this equation and letting $\vu=(\vv,\dx \vv)$ allows one to reduce the problem to a quasi-linear one. \\
%
%

Well posedness of the semi-linear analogues of \eqref{eq}, where the top order $a\equiv 1$, particularly when $b\equiv 0$ and $c$ is a polynomial, is quite well-understood with Local Smoothing and Strichartz estimates playing a significant role, cf.. \cite{KPV96} and \cite{KenStaf97} and references within. While quasi-linear dispersive equations are of interest in physical applications, in particular to water waves with variable dispersion, they are far less understood.\\

Well posedness has been established for quasilinear dispersive equations with special algebraic structure for which conservation laws have been found, for example the following shallow water wave equation
 \begin{align*}
   u_t- u_{txx}+ 3uu_x= 2u_xu_{xx}+ uu_{xxx}
 \end{align*}
see \cite{Const00} and references therein. However, finding such conservation laws is not possible in general for \eqref{eq}.\\

 A major advance for the well-posedness of a scalar \eqref{eq} was a work of Craig-Kappeler-Strauss in \cite{CKS92}. However, in addition to being restricted to scalar equation, they had to make a technical assumption for the favorable sign of the coefficient $b$, which was not natural in light of the semi-linear results in \cite{KenStaf97}. However, the pioneering work of Kenig-Ponce-Vega \cite{KPV2004} for the Quasilinear {\Schr} equation suggested a method to prove the well-posedness of \eqref{eq} under more general assumptions on the coefficients.\\

The method of Kenig et al. was a modification of the energy method, which was a successful approach to treat well-posedness of quasi-linear wave equation in high regularity Sobolev spaces in the 70's cf... \cite{HorLec}. Namely, the energy method relies on estimates of the form $\norm{\vu}_{H^s} = O(\norm{\vu_0}_{H^s})$, which are proved via integration by parts, symmetry of the top order terms, Sobolev embedding and Grownwall inequality. However, the energy method cannot in general work for the \eqref{eq} without modification due to the infinite speed of propagation. Overcoming this difficulty of controlling the size of a solution by the data occupies most of this paper. The heart of the matter can already be seen, when trying to prove $L^2$ well-posedness for a linear case of \eqref{eq}, regularized by a parabolic term for $0\le \ep \le 1$.
 \begin{equation}
  \label{eq:lin}
  \begin{cases}
    \dt\vu + a(x,t)\dx^3\vu+b(x,t)\dx^2\vu+c(x,t)\dx\vu+d(x,t)\vu=-\ep\dx^4\vu+\vf(x,t)\\
    \vu(x,0)=\vu_0(x)
  \end{cases}
\end{equation}
The standard energy method is available to prove an $L^2$ estimate for $t\ge 0$, only when $b\ge 0$. However, modifications of the energy method are needed, in general. Below we review the several works that motivated the approach used in this paper.\\

In the case of $a=Id$, $c$ symmetric and an integrable $b$, Kenig-Staffilani in \cite{KenStaf97}, motivated by \cite{HayOza94}, were able to cancel $b\dx^2$ term with a change of variables $v(x,t)=\Phi(x,t) u(x,t)$ for an appropriate $\Phi$. After this change of variables a standard energy argument (as explained above) works for $v$ and hence gives the $L^2$ estimate for $I=[0,T]$ small enough
\begin{equation}\label{est:lin}
\norm{\vu}_{\linf_{I} L^2_x} \le A(\norm{\vu_0(x)}_{L^2}+\norm{\vf(x,t)}_{L^1_{I}L^2_x})
\end{equation}
where $\norm{\vf(x,t)}_{L^1_{I}L^2_x} = \int_I \norm{\vf(t)}_{L^2_x} dt$ and likewise for all other space-time norms from now on. Similar argument was at the heart of the energy estimate that Lim-Ponce used to prove well-posedness of a quasilinear {\Schr} equation in 1 dimension in \cite{LimPon02}. However, this argument does not seem to work for a general system \eqref{eq:lin} with merely a symmetric top order $a\dx^3$, or for the {\Schr} equation in more than one space dimension. As we show in section \ref{seclin} the symmetry of the coefficient $a$ is sharp for the estimate \eqref{est:lin}, in the sense that without this assumption, this estimate can be false. This suggested to proceed via a Local Smoothing estimate argument, with the side benefit of capturing the regularization effect for equations \eqref{eq:lin} and \eqref{eq}.\\

 Kenig-Ponce-Vega in \cite{KPV2004} proved well-posedness of quasi-linear {\Schr} equation in higher space dimensions by proving a Local Smoothing estimate generalizing from the case of a time independent variable coefficient linear {\Schr} in \cite{Doi94} and \cite{CKSt95}.  Local Smoothing for \eqref{eq:lin} means that for $\delta >\frac{1}{2}$
 \begin{equation}\label{eq:gain}
 \norm{\frac{1}{\jap{x}^{\delta}}\,\jap{\dx}\vu}_{L^2_IL^2_x} \le A(\norm{\vu_0(x)}_{L^2}+\norm{\jap{x}^\delta\jap{\dx}^{-1} \vf(x,t)}_{L^2_{I}L^2_x})
 \end{equation}
 where we use the notation $\jap{x}=(1+\abs{x}^2)^{\frac{1}{2}}$ and interpret $\widehat{\jap{\dx} u}(\xi) = \jap{\xi}\hat u(\xi)$ as a Fourier multiplier and hence \eqref{eq:gain} means that the solution of \eqref{eq:lin} is $1$ derivative more regular than $u_0$ and $2$ than $f$ at a cost of weights. Note, that this effect is local, as the \eqref{eq:lin} is time reversible and by \eqref{est:lin} cannot gain smoothness globally.\\

 Local Smoothing was first proven for KdV in \cite{Ka83}, \cite{Kru83} and for the linear {\Schr} $\dt u + i L u =f$ with $L=\laplace$ in \cite{Const88}, \cite{Sjo87}, \cite{Ve88}, \cite{KPV93} where the gain is $\frac{1}{2}$ derivative relative to $u_0$ and $1$ derivative relative to $f$ respectively. This was generalized to $\mathcal{L}=a_{ij}(x)\partial_i\partial_j+b(x)\grad$ in \cite{CKSt95} and \cite{Doi94}. In \cite{Doi00} Doi showed that, roughly speaking, under appropriate asymptotic flatness of the coefficients for $\mathcal{L}$ above local smoothing is equivalent to the non-trapping of the bicharacteristic flow generated by the principal symbol of the operator $\mathcal{L}$. Note that in the one spacial dimension, which is the relevant setting for this paper, the non-trapping condition is automatic by the coefficient assumption (NL1) and we will omit it.\\

 Our proof of the wellposedness also involves using Local Smoothing to prove the energy estimate, however as \eqref{eq} is of higher order than the {\Schr} equation the argument of \cite{Doi94} has to be modified. Using this method to prove wellposedness of higher dispersive systems will be done in the subsequent work.\\

To motivate the function space we use to prove well-posedness, we note that in order to prove the main linear estimate \eqref{est:lin} a decay of the coefficient $b$ is necessary. This phenomenon is similar to the Mizohata condition, which shows the necessity of $\sup_{x,t\abs{\omega}=1}\int_0^t\Im b(x+s\omega)\cdot \omega ds <\infty$ for the $L^2$ well-posedness of {\Schr} equation $\dt u + i\laplace u + b(x)\grad u = 0$, cf... \cite{Miz85}, and we prove a corresponding result for \eqref{eq:lin} explicitly in the section \ref{seclin}. On the non-linear level of \eqref{eq} this suggests an $L^1$ condition on $\vu$ and weighted Sobolev spaces provide a natural way to ensure an $L^1$ condition in an $L^2$-based Sobolev space. This motivates working with the weighted Sobolev spaces $H^{s,2}$ for $s\in \Z^+$, which we define as follows:
\begin{equation}
  \label{def:norms}
  \begin{split}
& \norm{f(x)}_{H^{s, 2}} \equiv \sum_{j=0}^k\norm{ \jap{x}^{k-j} f(x)}_{H^{s+3j}} \approx_{s,2} \sum_{j=0}^2\sum_{\alpha=0}^{s+3j}\norm{ \jap{x}^{2-j} \dx^\alpha f}_{L^2}
\end{split}
\end{equation}
Where $\norm{f(x)}_{H^s} \equiv  \norm{(1+\abs{\xi}^2)^\frac{s}{2}\hat{f}(\xi)}_{L^2} \approx_s  \norm{\dx^s f(x)}_{L^2}+ \norm{f(x)}_{L^2}$. The space $H^{s,2}$ can be shown equivalent to $H^{s+10}_x\cap H^{s}(\jap{x}^4 dx)$, which is a typical choice to ensure extra decay of the lower order term $b\dx^2$ as explained above.

\subsection*{Coefficient assumptions} We state the precise assumptions on the coefficients $a$ -- $d$ of the equation \eqref{eq}:
\labelen{NL}
\begin{enumerate}
  \item \emph{Dispersive property.} The top coefficient $a$ is {\bf symmetric},  that is  \[ a_{ij}(x,t,\vec{z})=a_{ji}(x,t,\vec{z}) \text{ for all } (x,t,z) \in D_M \equiv \R\times[-1,1]\times \BM^{3n}.\]
      Moreover, it is {\bf uniformly positive definite} in $D_M$ at time $0$, that is  for every $M > 0$ there exists a constant $\lambda_M > 0$
\[    a_{ij}(x,0,\vec{z})\xi_i\xi_j \ge \lambda_M\abs{\xi}^2\]
whenever $(x,\vec{z}) \in \R\times \BM^{3n}$.
      \item \label{Regularity} \emph{Regularity.}  Let $J\in \Z^+$ be a given positive integer. We assume that all the coefficients $a(x,t,\vz)$-$d(x,t,\vz) \in C^1_t\B^J_x C^{J}_{z}$, with the space $\B^J$ of $C^J$ functions with all $J$ derivatives bounded.
           Specifically, there exists a function $(J,M)\to C_{J,M}$ increasing in $J$ and $M>0$ separately, such that for each $J$ and $M$ the coefficients
          \begin{equation*}
            a_{ij},\, b_{ij} \in C^1_t \B^J_{x,\vz}(D_M),\,c_{ij} \in C^1_tB^J_{x,\vz}(\R\times[-1,1]\times\BM^{2n}),
          \end{equation*}
with norms bounded by $C_{J,M}$, e.g. for $a$ that means
            \begin{align*}
              & \sup_{0\le \alpha \le 1;\,0\le\beta+\abs{\gamma} \le J}  \norm{\dt^\alpha\dx^\beta \partial_{\vz}^\gamma a(x,t,\vz)}_{\linf_{D_M}} \le C_{J,M}
            \end{align*}
   \item \emph{Asymptotic flatness and decay of linear parts.} There exists $\delta' > \frac{1}{2}$ and a constant $C_0$, such that
   \begin{align*}
  & \norm{\jap{x}^{2\delta'}\dx a(x,0,\vo)}_{\linf}+\norm{\jap{x}^{2\delta'}b(x,0,\vo)}_{\linf} \le C_0\\
    & \norm{\jap{x}^{2\delta'}\dt\dx a(x,t,\vo)}_{\linf_{\R\times I}}\!\!\!+\norm{\jap{x}^{2\delta'}\dt b(x,t,\vo)}_{\linf_{\R\times I}}\le C_0\\
    & \norm{\jap{x}^{\delta'}\dx b^\dag(x,0,\vo)}_{\linf}+\norm{\jap{x}^{\delta'}c^\dag(x,0,\vo)}_{\linf} \le C_0\\
    & \norm{\jap{x}^{\delta'}\dt \dx b^\dag(x,t,\vo)}_{\linf_{\R\times I}}\!\!\!\!+\norm{\jap{x}^{\delta'}\dt c^\dag(x,t,\vo)}_{\linf_{\R\times I}} \le C_0
    \end{align*}
    where $b^\dag$ and $c^\dag$ are the antisymmetric parts of $b$ and $c$ respectively, defined as $c^\dag_{ij} = \frac{1}{2} (c_{ij}-c_{ji})$
\end{enumerate}
\labelen{}
  While any $\delta' > \frac{1}{2}$ works in the (NL) assumptions, provided the definition of $H^{s,2}$ is modified with a weight $\jap{x}$ replaced by $\jap{x}^{\delta'}$, for simplicity we set $\delta' =1$.\\

  Note, that unlike the constants $\lambda_M$ in (NL1), which depend only on data $\vu_0$, $C_{J,M}$ in $(NL2)$ is a family of constants that depend on the smoothness of the coefficients $J$ and the size of the solution $M$. This is a natural assumption as we consider nonlinear equations and the coefficients may grow with the size of the solution. The size of the solutions is a priori unknown and is one of the quantities to be estimated. Thus when using (NL2) we will be explicit in the choice of $J$ and $M$ to avoid a possible circularity.

  The proof is based on the artificial viscosity regularization of the initial value problem \eqref{eq} for a parameter
 $0<\ep\le 1$:
  \begin{equation}\label{IVPe}
  \begin{cases}
    \partial_t \vu^\ep = -\ep\dx^4 \vu^\ep - N_{\vu^\ep}\vu^\ep + \vec{f}\\
    \vu^\ep(x,0)=\vu_0(x)
  \end{cases}
  \end{equation}
  Where $N_u$ is the operator from  \eqref{eq}
  \begin{equation}\label{Nu}
    \begin{split}
      & N_{\vu} =a(x,t,\vu,\dx\vu,\dx^2\vu)\cdot\dx^3 + b(x,t,\vu,\dx\vu,\dx^2\vu)
    \cdot\dx^2\\
    & + c(x,t,\vu,\dx\vu)\cdot\dx + d(x,t,\vu)\cdot I
    \end{split}
  \end{equation}
We then construct solutions $\vue$ for the regularized problem and show that the solution $\vu = \lim_{\ep\to 0} \vue$ in the desired topology.
\subsection*{Main Theorem}
\begin{theorem}\label{thm:nl}
Suppose the coefficients of \eqref{eq} satisfy assumptions (NL1)-(NL3) with bounds $\lambda_M$, $C_0$ and a function $C_{J,M}$ for $J\ge 19$. Then \eqref{eq} is \textbf{locally well-posed} in $H^{8,2}$. \
\end{theorem}
Specifically, by well-posedness we mean the following:
\begin{enumerate}
  \item \label{thm:nl:exist}\emph{Existence.} Let $R >0$ be given. Then there exists a $T>0$, such that if the data $        (\vu_0,\vf) \in Y\equiv H^{8,2}\times\left(L^1_{[-1,1]}H^{8,2}\cap C^0_{[-1,1]}H^{4,2}\right)$, such that
       \begin{equation}\label{eq:data}
\norm{(\vu_0,\vf)}_Y\equiv \norm{\vu_0}_{H^{8,2}}+\norm{\vf}_{L^1_{[-1,1]}H^{8,2}}+\norm{\vf}_{C^0_{[-1,1]}H^{4,2}} <R
      \end{equation}
  then \eqref{eq} has a solution $\vu\in X_I$, where $I=[0,T]$ and $X_I \equiv C^0_{I}H^{8,2}\cap C^1_I H^{4,2}$.
\item \label{thm:nl:uniq}\emph{Uniqueness.} $\vu$ is the unique solution of \eqref{eq} in $X_I$.
\item \label{thm:nl:cont} \emph{Continuous dependence.} The flow map $(\vu_0,\vf) \to \vu$ is continuous between $Y$ and $X_I$.
\item \label{thm:nl:pers}\emph{Persistence of regularity.} If $(\vu_0,\vf) \in Y\cap H^{s,2}\times L^1_{[-1,1]}H^{s,2}$ for $s>8$, then $\vu$ is in addition in $C^0_I H^{s,2}$.
\end{enumerate}
We make a few remarks about the Theorem \ref{thm:nl}.\\

We first note, that the time reversal $\vec{u}(x,t) \to \vec{u}(-x,-t)$ applied to \eqref{eq} preserves assumptions (NL1)--(NL3), thus Theorem \ref{thm:nl} extends to $[-T,0]$. The solution we construct has enough regularity to be classical and the uniqueness proved above is valid for $I= [-T,T]$. Thus Theorem \ref{thm:nl} produces a solution on $[-T,T]$. However, the parabolic regularization that we use to construct the solution requires a fixed sign of time and we stick to positive time intervals except for the uniqueness.\\

  The choice of well-posedness in $H^{8,2}$ and the choice of $T$ is determined by the main estimate Theorem \ref{lem:apri} that proves \eqref{est:lin} and its corollaries in the non-linear setting: Theorems \ref{thm:unif} and Proposition \ref{conv}.\\

   $H^{8,2}$ regularity for the well-posedness in \eqref{eq} is not sharp, but holds under very general assumptions on the coefficients and no smallness on data. While preparing this work for the publication, I have learned of a parametrix-based approach to the well-posedness of Quasi-linear {\Schr} for the small data in \cite{Marz11} in the lower regularity Sobolev spaces than in the \cite{KPV2004}. Adapting this approach to \eqref{eq} may allow to lower regularity for small data.\\



The argument of the proof of the Theorem \ref{thm:nl} also shows that \eqref{eq} has a Local Smoothing effect, that is 
   in addition to $X_I$, the solution is in $\vu \in L^2_I H^{9,2}(\jap{x}^{-4}dx)$.\\

A simple transformation $x\to -x$ in the equation \eqref{eq} shows that the dispersive property $(NL1)$ can taken with an opposite sign. That is, if in addition to (NL2)-(NL3), for every $M > 0$ there exists a constant $\lambda_M > 0$
\[    a_{ij}(x,0,\vec{z})\xi_i\xi_j \le -\lambda_M\abs{\xi}^2\]
whenever $(x,\vec{z}) \in \R\times \BM^{3n}$, then Theorem \ref{thm:nl} holds.\\

 We note that the continuous dependence in the Theorem \ref{thm:nl} is the best we can hope for as \eqref{eq} is quasilinear.\\

 Finally, the arguments in the proof of Theorem \ref{thm:nl} can be used to extend the persistence of regularity to $H^{s,k}$ rather than $H^{s,2}$.\\

 The rest of the paper is organized as follows. In the section \ref{seclin} we prove the main linear  estimate \eqref{est:lin}. In the section \ref{sec:parab} we prove the well-posedness of \eqref{IVPe} for a time independent of the regularization $\ep$. Finally, in the section \ref{seclim} we construct the solution of \eqref{eq} and prove continuous dependence.

\subsection*{Notation}
When estimating with multiplicative constants, we often write $A\les_{x,y} B$, to mean $A \le C(x,y) B$, where the constant $C(x,y)$ may change from line to line.\\
 As dependence on the integers $(8,2)$ and the constant in the Sobolev embedding $H^1(\R) \hookrightarrow \linf(\R)$ occurs frequently we do not explicitly state dependence on them.

    \section{Linear estimate}\label{seclin}
In this section we prove the estimate \eqref{est:lin} for the linear equation \eqref{eq:lin} and show the necessity of several coefficient assumption. Main estimates for the well-posedness in sections \ref{secuni} and \ref{seclim} are later reduced to it.\\

Let $T'>0$ and suppose \eqref{eq:lin} is valid on $[0,T']$. We assume that there exist constants $\lambda>0$, $C_1\ge \tld C_0>0$ and $\delta >\frac{1}{2}$, such that the coefficients $a$, $b$, $c$, $d$ of the equation \eqref{eq:lin} satisfy the following conditions on $\R\times I\equiv \R\times [0,T']$:
\labelen{L}
\begin{enumerate}
  \item \emph{Dispersive property.} Assume that the top coefficient is symmetric,  that is $a_{ij}(x,t)=a_{ji}(x,t)$. Moreover, it is uniformly positive definite with
\[    a_{ij}(x,0)\xi_i\xi_j \ge \lambda\abs{\xi}^2\]
uniformly in $x$ in $\R$ and $\xi\in \Rn$.
      \item \emph{Regularity.}  $a\in C^1_{I}\B^3_x$, $b \in C^1_{I} \B^2_x$, $c\in C^1_{I} \B^1_x$ and $d \in C^1_{I} \linf_x$ with the space $\B^k$ of bounded $C^k$ functions have the following bounds: 
          \begin{align*}
            & \norm{ a(x,0)}_{\B^3} + \norm{b(x,0)}_{\B^2} + \norm{c(x,0)}_{\B^1} + \norm{d(x,0)}_{\linf} \le \tld C_0\\
            & \norm{ a(x,t)}_{C^1_{I}\B^3} + \norm{b(x,t)}_{C^1_{I}\B^2} + \norm{c(x,t)}_{C^1_{I}\B^1} + \norm{d(x,t)}_{C^1_{I}\linf} \le C_1
          \end{align*}

   \item \emph{Asymptotic flatness and decay.}
   \begin{align*}
  & \norm{\jap{x}^{2\delta}\dx a(x,0)}_{\linf}+\norm{\jap{x}^{2\delta}b(x,0)}_{\linf} \le \tld C_0\\
    & \norm{\jap{x}^{2\delta}\dt\dx a(x,T)}_{\linf_{\R\times I}}\!\!\!+\norm{\jap{x}^{2\delta}\dt b(x,t)}_{\linf_{\R\times I}}\le C_1\\
    & \norm{\jap{x}^{\delta}\dx b^\dag(x,0)}_{\linf}+\norm{\jap{x}^{\delta}c^\dag(x,0)}_{\linf} \le \tld C_0\\
    & \norm{\jap{x}^{\delta}\dt \dx b^\dag(x,t)}_{\linf_{\R\times I}}\!\!\!\!+\norm{\jap{x}^{\delta}\dt c^\dag(x,t)}_{\linf_{\R\times I}} \le C_1
    \end{align*}
    where $b^\dag$ and $c^\dag$ are the antisymmetric parts of $b$ and $c$ respectively, defined as $c^\dag_{ij} = \frac{1}{2} (c_{ij}-c_{ji})$
\end{enumerate}
\labelen{}

\begin{theorem}
  \label{lem:apri}
There exist constants $A=A(\tld C_0,\lambda,\delta)$ and $T=T(\tld C_0,C_1,\lambda,\delta)\le 1$, such that if $\vu(x,t)$
  is a solution of \eqref{eq:lin} on $I=[0,T'] \subset [0,T]$ , then $\vu$ satisfies \eqref{est:lin}, i.e.
  \[
  \norm{\vu}_{\linf_{I} L^2_x} \le A(\norm{\vu_0(x)}_{L^2}+\norm{\vf(x,t)}_{L^1_{I}L^2_x})
  \]
  and \eqref{eq:gain}.
%
\end{theorem}
By a solution of \eqref{eq:lin} we mean a classical solution and hence by the Sobolev embedding $C^0_IH^5\cap C^1_I H^1$ would suffice.\\

Regularity of $a$-$d$, particularly with weights in (L3) determines the Sobolev exponent of $H^{8,2}$ in Theorem \ref{thm:nl}, where for simplicity, we set $\delta=\delta'=1$. Proof of analogues of Theorem \ref{lem:apri} with coefficients rougher than above leads to lowering regularity in Theorem \ref{thm:nl}.\\

  Note, that for the applications of Theorem \ref{lem:apri}, the constant $C_1$ will depend on the solution of the non-linear problem, while $\tld C_0$ will only depend on data. As we will use the constant $A=A(\lambda,\tld C_0)$ from Theorem \ref{lem:apri} to control the size of the solution and in turn $C_1$, it is crucial for $A$ not to depend on $C_1$.\\

It is not difficult to prove an $H^s$ version of \eqref{est:lin} and \eqref{eq:gain}, provided coefficients are more regular than (L2), by differentiating \eqref{eq:lin}, using the Theorem \ref{lem:apri} and choosing $T$ small to control lower order terms. However, for some estimates in the proof of Theorem \ref{thm:nl} we will not have such control of the coefficients, and we omit the $H^s$ estimate.

\begin{remark}\label{rem:lem:sign}
For $\ep=0$, Theorem \ref{lem:apri} is valid with $I=[-T',0]$ or $[-T',T']$ and $C^1_IH^1\cap C^0 H^4$ regularity.
\end{remark}
Remark follows from Theorem \ref{lem:apri} by a simple scaling of the equation. Indeed, sending $(x,t)\to (-x,-t)$ preserves assumptions (L1)-(L3), while changing the sign of the time.\\

Finally before proceeding with the proof of the Theorem \ref{lem:apri} in the subsection \ref{sec:lin:proof}, we motivate the coefficient assumptions (L1)--(L3) by showing the necessity of symmetry of $a$ and decay of $b$ for the Theorem \ref{lem:apri}.
\subsection{Symmetry of the top order}\label{sec:sym}
Similar to the hyperbolic systems, symmetry of \eqref{eq} is necessary for the well-posedness. Namely, we consider the following constant coefficient linear system that violates the symmetry in (NL1), but satisfies all other assumptions:
\begin{equation}\label{nonsym}
\begin{cases}
  & \dt \left(
          \begin{array}{c}
            u_1 \\
            u_2 \\
          \end{array}
        \right)               +\left(
                               \begin{array}{cc}
                                1 &  \delta\\
                                0 & 1 \\
                                \end{array}
                                    \right)        \dx^3 \left(
                                                    \begin{array}{c}
                                                    u_1 \\
                                                    u_2 \\
                                                    \end{array}
                                                    \right)         =\left(
                                                                  \begin{array}{c}
                                                                    0 \\
                                                                    0 \\
                                                                  \end{array}
                                                                    \right), \text{ where } \delta \neq 0\\
   &\vu(x,0)^T=\begin{pmatrix}
     u_{0,1}(x) & u_{0,2}(x)
   \end{pmatrix}
\end{cases}
\end{equation}
Taking a Fourier transform in space this equation reduces to an ODE, 
for which the explicit solution is
\begin{align*}
\hat \vu(\xi,t)^T = \begin{pmatrix}
  e^{i\xi^3t}(\hat{u}_{0,1}(\xi) + \ep i \xi^3t \hat{u}_{0,2}(\xi)) & e^{i\xi^3t}\hat{u}_{0,2}(\xi)
\end{pmatrix}
\end{align*}
We then take the data $\hat\vu_0(\xi)^T = \left(
                      \begin{array}{cc}
                        0 & \jap{\xi}^{-s-1}
                      \end{array}
                    \right)
$ and a computation shows\\ $\norm{\vu_0}_{H^s} = \sqrt{\pi}$, while for any $t\neq 0$, $\norm{\vu(t)}_{H^s_x} = \infty$. Therefore, \eqref{nonsym} is ill-posed in $H^s$ for any $s$ and hence in any $H^{s,2}$.


%
\subsection{Necessity of decay of the coefficient $b$}
Here we show that a slightly weaker form of asymptotic flatness (L3) is necessary for \eqref{est:lin} in the following special case. More precisely, we show that for a solution $u$ of
\begin{align}\label{eq:dec}
  \begin{cases}
     \dt u + \dx^3 u+b(x)\dx^2 u=f\\
     u(x,0)= u_0(x)
  \end{cases}
\end{align}
the condition on the coefficient $b$
\begin{equation}
  \label{decay}
  \sup_{(x,t,r)} \int_0^t b(x+ r \cdot s )\, r ds <\infty
\end{equation}
\textbf{is necessary for \eqref{est:lin}.}\\

We show necessity by a WKB method, similar to an argument of \cite{Miz85} for a \Schr equation. Let $u=e^{i\phi(x,t,\xi)} v(x,t,\xi)$, where $\phi = x\xi^2 + t\xi^6$. Then for an operator $L = \dt+\dx^3+b(x)\dx^2$ we compute
\begin{align*}
  e^{-i\phi} L(e^{i\phi} v) = (\dt v - 3\xi^4 \dx v - b(x)\xi^4 v) +\left[3i\xi^2 \dx^2v + \dx^3 v + 2 b(x)i\xi^2 \dx v + b(x)\dx^2 v\right]
\end{align*}
We set $v$ to be a solution of the transport equation
\[
\begin{cases}
  \dt v - 3\xi^4 \dx v - b(x)\xi^4 v = 0\\
  v(x,0) = v_0(x)
\end{cases}
\]
For which we get the explicit solution by the method of characteristics
\begin{align*}
  v(x,t,\xi) = e^{\int_0^t b(x+3\xi^4 s) \xi^4 ds} v_0(x+3\xi^4 t)
\end{align*}
We further define
\begin{align*}
  f = e^{-i\phi} \left[3i\xi^2 \dx^2v + \dx^3 v + 2 b(x)i\xi^2 \dx v + b(x)\dx^2 v\right]
\end{align*}
Thus $u$ solves \eqref{eq:dec}, $\norm{u}_{L^2}=\norm{v}_{L^2}$ and $\norm{u_0}_{L^2}=\norm{v_0}_{L^2}$.\\

Now assume that \textbf{\eqref{decay} does not hold}. Then there exists a $x_0$, $t_0$ and $r_0$ such that
\begin{align*}
  \int_0^{t_0} b(x_0+ r_0 \cdot s ) r_0 ds \ge 6\log (3A)
\end{align*}
Moreover we can assume by rescaling $t_0$ with $\frac{t_0}{r_0}$ that $r_0=1$. Hence for $x$ near $x_0$
\begin{align}\label{decay:not}
  \int_0^{t_0} b(x+  s )  ds \ge 6\log (2A)
\end{align}
Now define $t_0^\xi = \frac{t_0}{3\xi^4}$ and likewise $t^\xi$ for any $t$.
\begin{equation}
  \label{transp}
  v(x,t^\xi,\xi) =e^{\int_0^{t^\xi} b(x+3\xi^4 s) \xi^4 ds} v_0(x+3\xi^4 t^\xi)= e^{\frac{1}{3}\int_0^{t} b(x+ s) ds} v_0(x+t)
\end{equation}
Now let $v_0 \in C_0^\infty(\R)$ be a function compactly supported near $x_0-t_0$ with $\norm{\vv_0}_{L^2} = 1$. Then
\begin{align*}
  \norm{\vv(x,t_0^\xi,\xi)}_{L^2_x}^2 \ge \int_{x \approx x_0} (2A)^2 \abs{\vv_0(x)}^2 dx = (2A)^2
\end{align*}
Moreover, by \eqref{transp} and if \textbf{$b$ in $\B^3$} (i.e. $C^3$, bounded and with bounded derivatives) we show that  $v(x,t^\xi,\xi)$, \ldots $\dx^3v(x,t^\xi,\xi)$ are uniformly bounded in $L^2_x$ for  $t^\xi \in I_{\xi} = \left[ -\abs{t_0^\xi},\abs{t_0^\xi}\right]$. Therefore,  $
  \norm{f}_{L^1_{I_\xi}L^2_x} \les \int_{-\abs{t_0^\xi}}^{\abs{t_0^\xi}} \abs{\xi}^2 dt = o(1)$, as $\xi \to \infty$. However,  \eqref{est:lin} and calculations above imply that $  2 A \le \norm{\vu}_{\linf_{I_\xi} L^2_x} \le A (1 + o(1))$, which is a contradiction for $\xi$ large enough.
\subsection{Proof of Theorem \ref{lem:apri}}

\label{sec:lin:proof}

To proceed we first change the dependent variable (or gauge it) and then proceed with the energy method for the gauged problem.\\

 We use $\jap{x}^{-2\delta}$ to define a multiplicative change of variables as follows:
 \begin{equation}\label{phi}
 \phi(x)=-N\int_{-\infty}^x  \jap{x'}^{-2\delta} dx' \text{ with }N=\frac{2+10n\tld C_0}{3\lambda}
 \end{equation}
 Because the integral above is convergent and $\abs{\dx^\alpha \jap{x}^\beta} \le C(\alpha,\beta) \jap{x}^{\beta-\alpha}$ we get $\phi \in \B^\infty$, i.e. it is bounded and all of its derivatives are bounded. Moreover,
 \[
 \dx\phi(x) = -N\jap{x}^{-2\delta} \le 0
 \]
 We now define the gauged variable
  \begin{align}\label{gauge}
   \vv(x,t)=e^{-\phi(x)}\vu(x,t),
 \end{align}
Definition of $\phi$ implies that $e^{\phi}$ and $e^{-\phi}$ are also in $\B^\infty$, and hence
 \[
 \norm{\vv}_{L^2} = \norm{e^{-\phi}\vu}_{L^2} \approx_{N,\delta} \norm{\vu}_{L^2_x}
 \]
 Likewise, by the product rule and duality,
 \begin{align}\label{eq:compar}
   \norm{\jap{x}^k\vu}_{H^{s}} \approx_{N,\delta,k} \norm{\jap{x}^k\vv}_{H^{s}} \approx_{s,k} \norm{\jap{x}^k \jap{\dx}^s\vv}_{L^2}
 \end{align}
 for $s=\pm 1$.\\

 Inverting \eqref{gauge} we write \eqref{eq:lin} as:
 \begin{align*}
    & \dt(e^{\phi(x)}\vv) + a(x,t)\dx^3(e^{\phi(x)}\vv)+b(x,t)\dx^2(e^{\phi(x)}\vv)+c(x,t)\dx(e^{\phi(x)}\vv)\\
    & +d(x,t)e^{\phi(x)}\vv = -\ep\dx^4(e^{\phi(x)}\vv)+\vf(x,t)
 \end{align*}
 Then $\vv$ satisfies the following "gauged" system
 \begin{equation}\label{eq:gauge}
   \begin{cases}
     & \dt\vv + L\vv=-\ep\dx^4\vv+\ep R_3\vv+\vg(x,t)\\
    & \vv(x,0)=\vv_0(x)
   \end{cases}
 \end{equation}
 where
 \begin{align*}
      & \vg(x,t)=e^{-\phi(x)}\vf(x,t) \text{ and } \vv_0(x) = e^{-\phi(x)}\vu_0(x)
 \end{align*}
 and the operator $L = a(x,t)\dx^3+\tb(x,t)\dx^2+\tc(x,t)\dx+\td I$
has the coefficients
 \begin{align*}
   & \tb_{ij}(x,t) = b_{ij}(x,t)+3a_{ij}(x,t)\dx\phi(x)\\
   & \tc_{ij}(x,t) = c_{ij}(x,t)+2b_{ij}(x,t)\dx\phi(x)+3a_{ij}(x,t)([\dx\phi]^2 + \dx^2\phi)\\
   & \td_{ij}(x,t) = d_{ij}(x,t)+c_{ij}(x,t)\dx\phi(x)+b_{ij}(x,t)([\dx\phi]^2 + \dx^2\phi)\\
    & + a_{ij}(x,t)([\dx\phi]^3+3\dx^2\phi\dx\phi+\dx^3\phi)
 \end{align*}
 and $R_3=[-\dx^4,e^\phi]$ is a third order differential with $\B^\infty$ coefficients with norms controlled by $N$ and $\delta$.\\

 We now proceed with the energy estimates. Taking a dot product of \eqref{eq:gauge} by $\vv$ and integrating in $x$ we get
 \begin{align}\label{gge1}
    \int \dt\vv\cdot\vv \,dx + \int L\vv\cdot\vv \,dx =-\ep\int\dx^4\vv\cdot\vv\, dx +\ep \int R_3\vv\cdot\vv \,dx +\int \vg\cdot\vv \,dx
 \end{align}
 Note that $L^*$, the adjoint of the operator $L$, is $L^*\vv = - a^T \dx^3\vv + l.o.t.$, where $a^T_{ij}=a_{ji}$ and $l.o.t.$ are terms having less than $3$ derivatives of $\vv$. Hence integration by parts of \eqref{gge1} gives
 \begin{equation}\label{gge2}
   \begin{split}
     \tfrac{1}{2}\dt \norm{\vv}_{L^2}^2 = \int B\dx \vv\cdot \dx \vv + \int C\dx \vv \cdot \vv + \int D\vv \cdot \vv \\ - \ep\int\dx^4\vv\cdot\vv +\ep \int R_3\vv\cdot\vv +\int \vg\cdot\vv \equiv \sum_{j=1}^6 I_j
   \end{split}
 \end{equation}
 where
 \begin{align*}
   & B_{ji}= B_{ij} = -3\dx a_{ji}+b_{ij}+b_{ji}+6\dx \phi a_{ij}\\
   & -C_{ji}=C_{ij}=\dx (b_{ij}-b_{ji}) + (c_{ij}-c_{ji}) + 2(b_{ij}-b_{ji})\dx\phi\\
    & D_{ij} = -\dx^3a_{ji}+\dx^2\tb_{ji}-\dx \tc_{ji}+\td_{ij}+\td_{ji}
 \end{align*}
We now claim that the Theorem \ref{lem:apri} reduces to the following proposition.


    \begin{proposition}\label{apri:dt}
There exist $A=A(\tld C_0, \lambda, \delta)$ and $T=T(C_1,\tld C_0, \lambda, \delta)$, such that for $0\le t\le T'\le T\le 1$, the solution of \eqref{eq:gauge} satisfies:
   \begin{equation}\label{eq:dt}
\dt\norm{\vv(x,t)}_{L^2_x}^2 + \int \jap{x}^{-2\delta}\abs{\dx\vv}^2dx \le A   \norm{\vv(x,t)}_{L^2_x}^2 + 2\abs{\int \vg\cdot \vv dx}
   \end{equation}
 \end{proposition}
 Indeed, discarding the second term on the right hand side of \eqref{eq:dt} and the Grownwall inequality imply
 \begin{align}\label{eq:dt2}
    \norm{\vv(x,t)}_{L^2_x}^2 \le e^{At}( \norm{\vv_0}_{L^2}^2 + \int_0^t \abs{\int \vg \cdot \vv dx} dt)
 \end{align}
 Thus by the Cauchy-Schwartz inequality we get for $I=[0,T']$
 \begin{align*}
   \norm{\vv}_{L^1_I L^2_x} \le e^{At}(\norm{\vv_0}_{L^2} + \norm{\vg}_{L^1_IL^2_x})
 \end{align*}
 which is \eqref{eq:lin} after we use the comparability of the norms of $\vu$ and $\vv$.\\
Integrating \eqref{eq:dt} and using \[\int \vg \cdot \vv dx = \int \jap{x}^{\delta}\jap{\dx}^{-1}\vg\cdot \jap{\dx}(\jap{x}^{-\delta}\vv) dx\]
 in \eqref{eq:dt2} gives the \eqref{eq:gain} after the comparability of the norms.
\subsection{Coefficient estimates}
Before proving the Proposition \ref{apri:dt}, we first rephrase (L1)-(L3) in the following form
\begin{lemma}\label{claim:T}
  There exists $T =T(\tld C_0,\lambda,C_1)$ small enough, such that for $0\le t\le T'\le T$
  \begin{description}
    \item[L'1]
    \[
     a_{ij}(x,t)\xi_i\xi_j \ge \frac{\lambda}{2}\abs{\xi}^2
    \]
    uniformly in $x$.
  \item[L'2]
      \begin{align*}
            & \norm{ a(x,t)}_{C^0_{I}\B^3} + \norm{b(x,t)}_{C^0_{I}\B^2} + \norm{c(x,t)}_{C^0_{I}\B^1} + \norm{d(x,t)}_{C^0_{I}\linf} \le 2 \tld C_0
          \end{align*}
   \item[L'3]
   \begin{align*}
  & \norm{\jap{x}^{2\delta}\dx a(x,t)}_{\linf}+\norm{\jap{x}^{2\delta}b(x,t)}_{\linf} \le 2 \tld C_0\\
    & \norm{\jap{x}^{\delta}\dx b^\dag(x,t)}_{\linf}+\norm{\jap{x}^{\delta}c^\dag(x,t)}_{\linf} \le 2\tld C_0\\
    \end{align*}
    \end{description}
\end{lemma}
 \begin{proof}[Proof of Lemma \ref{claim:T}]
By the Fundamental theorem of Calculus, (L1) and (L2)

\[
 a_{ij}(x,t)\xi_i \xi_j = \left[a_{ij}(x,0)+\int_0^t\dt a_{ij}(x,t')dt'\right]\xi_i\xi_j
 \ge \lambda\abs{\xi}^2 - nTC_1\abs{\xi}^2
\]
Which implies (L'1) for $T$ small enough.\\

(L'2) and (L'3) are established similarly.
 \end{proof}
 \begin{lemma}\label{lem:coeff}
   With $0\le t\le T$ as in Lemma \ref{claim:T}
   \begin{align}\label{eq:B}
     \jap{x}^{2\delta} B_{ij}(x,t)\xi_i\xi_j \le - 2\abs{\xi}^2
   \end{align}
   Moreover, there exists a constant $A =A(\tld C_0,\lambda,\delta)\ge 0$, such that
   \begin{align}
\label{eq:C}     \norm{\jap{x}^{\delta} C(x,t)}_{\linf_x} \le \sqrt{A}\\
\label{eq:D}     \norm{D(x,t)}_{\linf_x} \le A
   \end{align}
 \end{lemma}
 \begin{proof}[Proof of Lemma \ref{lem:coeff}]
By the definition of $B$ in \eqref{gge2}, (L'1) and (L'3):
 \begin{align*}
  & \jap{x}^{2\delta} B_{ij}(x,t)\xi_i\xi_j \le 6\jap{x}^{2\delta} \dx \phi(x) a_{ij}(x,t)\xi_i\xi_j\\
   & + n\jap{x}^{2\delta} \sup_{ij}(3\abs{\dx a_{ij}(x,t)}+\abs{b_{ij}(x,t)+b_{ji}(x,t)})\abs{\xi}^2\\
  & \le  (-3 N \lambda + 10n\tld C_0 )\abs{\xi}^2
\end{align*}
Which proves \eqref{eq:B} by the choice of $N$ in \eqref{phi}.\\
Definition of $C$ in \eqref{gge2}, (L'3) and \eqref{phi} imply \eqref{eq:C}. \eqref{eq:D} is done similarly.
 \end{proof}
\begin{proof}[Proof of Proposition \ref{apri:dt}]
We estimate \eqref{gge2} term by term in reverse order:\\
Trivially,
\[
I_6 \le \abs{\int \vv\cdot\vg dx}
\]
For $I_5$, by the Calculus of Pseudo-Differential Operators $\jap{\dx}^{-\frac{3}{2}}R_3$ is a {$\Psi$DO} of order $\frac{3}{2}$ and hence maps $H^{\frac{3}{2}}$ to $L^2$ by the boundedness of Pseudodifferential operators. Thus
 \begin{align*}
   &I_5= \ep \int  \jap{\dx}^{-\frac{3}{2}} R_3\vv\cdot\jap{\dx}^{\frac{3}{2}} \vv dx  \les_{N,\delta} \ep \norm{\vv}_{H^{\frac{3}{2}}}^2
 \end{align*}
 Hence interpolation of $H^{\frac{3}{2}}$ between $L^2$ and $H^2$, Cauchy inequality and \eqref{def:norms} give
 \begin{align*}
   I_5 \le \ep\norm{\vv}_{H^2}^2 + A\norm{\vv}_{L^2}^2 \le \ep\int \abs{\dx^2\vv}^2 dx + A\norm{\vv}_{L^2}^2
 \end{align*}
 While for $I_4$, integrating by parts gives
 \begin{align*}
   I_4 = -\ep\int \dx^4\vv \cdot \vv dx = -\ep\int \abs{\dx^2\vv}^2 dx
 \end{align*}
We use \eqref{eq:D} for $I_3$
\begin{align*}
  I_3 = \int D\vv\cdot \vv \le \norm{D(t)}_{\linf_{x}}\norm{\vv}_{L^2}^2 \le A\norm{\vv}_{L^2}^2
\end{align*}
By Cauchy-Schwartz and \eqref{eq:C} we estimate
\begin{align*}
    I_2 = \int \frac{\dx \vv}{\jap{x}^\delta} \cdot \jap{x}^\delta C^T \vv
    \le \norm{\frac{\dx \vv}{\jap{x}^\delta}}_{L^2_x}\norm{\jap{x}^\delta C^T(t)}_{\linf}\norm{\vv}_{L^2}\\
    \le \norm{\frac{\dx \vv}{\jap{x}^\delta}}_{L^2_x}^2+A\norm{\vv}_{L^2}^2
\end{align*}
where $C^T_{ij}=C_{ji}$.\\ Finally, by \eqref{eq:B}
\begin{align*}
  I_1 = \int \jap{x}^{2\delta}B\frac{\dx \vv}{\jap{x}^\delta}\cdot \frac{\dx \vv}{\jap{x}^\delta} dx
  \le - 2\int  \frac{\abs{\dx \vv}^2}{\jap{x}^{2\delta}} dx
\end{align*}
Summing together the estimates of $I_1$ -- $I_6$ we complete the proof of the Proposition \ref{apri:dt} and hence the Theorem \ref{lem:apri}
 \end{proof}

\section{Wellposedness of the regularized problem.}\label{sec:parab}
 By Duhamel principle, solving \eqref{IVPe} in a sufficiently regular Sobolev space is equivalent to a fixed point of the operator
 \begin{align}\label{IVPe:Duh}
  \Gamma\vue \equiv e^{-\ep t\dx^4}\vec{u}_0 + \int_{0}^t e^{-\ep (t-t')\dx^4}(-N_{\vu^\ep}\vu^\ep+\vf)(t')dt'
 \end{align}
 where the parabolic operator $e^{-\ep t\dx^4}$ is defined as a Fourier multiplier:
    \begin{align*}
\widehat{e^{-\ep t\dx^4}\vec{u}_0}(\xi) = e^{-\ep t \xi^4}\widehat{\vu}_0(\xi)
    \end{align*}
    Well-posedness of \eqref{IVPe} for a short time dependent on the parameter $\ep$ is very similar to \cite{KPV2004}, where the same regularization is used for a quasi-linear {\Schr} equation. However, as \eqref{eq} is of higher order than {\Schr}, we provide the proof in the Proposition \ref{regul}.
    \subsection{Preliminary estimates}
    \begin{lemma}
      \label{lem:par} For any positive integer $s$, $0<T\le 1$ and $t\in I=[0,T]$ the following estimates hold
      \begin{subequations}\label{est:par}
   \begin{align}
        & \norm{ e^{-\ep t\dx^4}\vec{u}_0}_{C^0_IH^{s,2}}\le C(s) \norm{\vu_0}_{H^{s,2}}
        \label{est:par1}\\
        & \norm{\int_0^t e^{-\ep (t-t')\dx^4}\,\vf(t') dt'}_{C^0_IH^{s,2}} \le C(s) \left(T+\sqrt[4]{\frac{T}{\ep^3}}\right)\norm{f}_{ C^0_I H^{s-3,2}}\label{est:parg}
   \end{align}
    \end{subequations}
    \end{lemma}
    By changing the $Y$ norm in \eqref{eq:data} by a multiplicative constant, for the rest of the paper we treat the constant $C(8)$ as $1$ from the Lemma \ref{lem:par} for simplicity.
    \begin{proof}
      By Plancherel and boundedness of $e^{-\alpha}$ for $\alpha\ge 0$:
      \begin{align}\label{est:par1a}
      \norm{e^{-\ep t\dx^4}\vec{u}_0}_{H^s}^2 = \int \jap{\xi}^{2s} e^{-2\ep t\xi^4}\abs{\widehat{\vec{u}_0}(\xi)}^2 d\xi \le \norm{\vu_0}_{H^s}^2
      \end{align}
      To proceed with weights, note that multiplying the PDE for $\vv = e^{-\ep t\dx^4}\vec{u}_0$ by $\jap{x}^2$ and commuting derivatives with weights, we get
      \begin{align*}
      \begin{cases}
        \dt \jap{x}^2 \vv = - \ep \dx^4 (\jap{x}^2\vv) +  E_3 (\jap{x}\vv)\\
        \jap{x}^2\vv(x,0) = \jap{x}^2\vu_0(x)
      \end{cases}
      \end{align*}
      Where $E_3=[\ep \dx^4,\jap{x}^2]\jap{x}^{-1}$ is a differential operator of order $3$ with $\B^\infty$ coefficients. Hence by the Duhamel formula
      \begin{align}\label{comm:par}
        \jap{x}^2e^{-\ep t\dx^4}\vec{u}_0 = e^{-\ep t\dx^4}\jap{x}^2\vec{u}_0 + \int_0^t e^{-\ep (t-t')\dx^4}E_3 \jap{x} \vv dt'
      \end{align}
      Therefore, by \eqref{est:par1a} and Minkowski inequality
      \begin{align*}
        \norm{\jap{x}^2e^{-\ep t\dx^4}\vec{u}_0}_{H^s} \le \norm{\jap{x}^2\vu_0}_{H^s} + \int_0^t \norm{e^{-\ep (t-t')\dx^4}E_3 \jap{x} \vv}_{H^s} dt'\\
\le        \norm{\jap{x}^2\vu_0}_{H^s} + t\norm{E_3\jap{x}\vv}_{H^s}
      \end{align*}
      By the boundedness of (Pseudo)Differential operators, $\norm{E_3 \jap{x}\vv}_{H^s} \les_s \norm{\jap{x}\vv}_{H^{s+3}}$.
      By Pseudo-Differential calculus (or commuting $\jap{x}$ with derivatives by hand) and Cauchy-Schwarz
      \begin{equation}\label{wt:int}
\norm{\jap{x}\vv}_{H^{s+3}} = \int \jap{x}\jap{\dx}^s\jap{x}\vv\cdot \jap{x}^{-1}\jap{\dx}^{s+6}\jap{x}\jap{\dx}^s\vv dx  \les_s \norm{\jap{x}^2 f}_{H^s}^{\frac{1}{2}}\norm{ f}_{H^{s+6}}^{\frac{1}{2}}
\end{equation}
Thus by \eqref{def:norms}
\begin{align*}
  \norm{e^{-\ep t\dx^4}\vec{u}_0}_{H^{s,2}} = \norm{\vv}_{H^{s,2}} \les_s \norm{\vu_0}_{H^{s,2}}
\end{align*}
which is \eqref{est:par1}.\\

Using elementary Calculus estimate $\alpha^{\frac{3}{2}} e^{-\alpha} \le \frac{1}{e}$ for $\alpha \ge 0$ and the explicit formula for the semigroup $e^{-\ep t\dx^4}$, $t\ge 0$ we get
\begin{align*}
  \norm{\int_0^t e^{-\ep (t-t')\dx^4}\,\vf(t') dt'}_{C^0_IH^{s}} \les_s \left(T+\sqrt[4]{\frac{T}{\ep^3}}\right) \norm{f}_{ C^0_I H^{s-3}}
\end{align*}
Using this estimate instead of \eqref{est:par1a} with weights finishes the proof.
    \end{proof}
        We also need the following estimate for $N(\vu)\equiv N_{\vu}(\vu)$ from \eqref{Nu}:
    \begin{lemma}[Moser estimate]\label{lem:Mos}
            Let $u$, $v \in C^0_I H^{s,2}$ for $s=8$ or $9$ be given with
            \[\max\{\norm{\vu}_{C^0_I H^{8,2}},\norm{\vv}_{C^0_I H^{8,2}}\}\le M\]
             Then
    \begin{subequations}\label{est:nl}
    \begin{align}
       & \norm{N(\vu)}_{ C^0_I H^{s-3,2}} \le C_{s+3,M}(1+M^{s+3}) \norm{\vu}_{C^0_IH^{s,2}}\label{est:nl1}\\
        & \norm{N(\vu)-N(\vv)}_{ C^0_I H^{s-3,2}} \le C_{s+3,M}(1+M^{s+3}) \norm{\vu-\vv}_{C^0_IH^{s,2 }}\label{est:nld}
    \end{align}
    \end{subequations}
    \end{lemma}
\begin{proof}
The proposition follows by the elementary calculus and the Sobolev embedding. The constants in the $s=9$ case depend on $M$, because it is impossible to get terms like $\dx^{s+5}\vu\cdot \dx^{s+6}\vu$ by differentiating $N(\vu)$ $s+3$ times.
\end{proof}
      \subsection{Short time well-posedness}
 We now set up the following notation:
 \begin{equation}\label{M}
  \begin{split}
     M = (16 A +2)R \text{ and } T_\ep =  \min \left\{\frac{\ep^3}{[4C_{11,M}(1+M^{11})]^4}, \frac{1}{2}\right\}
  \end{split}
\end{equation}
 with $R$ from \eqref{eq:data} and $A$ from the Theorem \ref{thm:unif}.

 We do the contraction argument in the following closed subset of $C^0_IH^{8,2}$ for $I=[0,T]$. 
\begin{equation}
  \label{XT}
    X_{I}^M = \{ \vu \in C^0_{I}H^{8,2}: \vu(x,0)=\vu_0(x),\, \norm{\vu}_{C^0_{I}H^{8,2}} \le M, \norm{\dt \vu}_{C^0_{I}H^{4,2}}\le C(M) \}
\end{equation}
with $\norm{\vu-\vv}_{X^M_I} = \norm{\vu-\vv}_{C^0_I H^{8,2}}$ and
\begin{align*}
      C(M) = \sup_{\norm{\vu}_{C^0_{I}H^{8,2}} \le M}\norm{\vf -N_{\vu}(\vu)}_{C^0_IH^{4,2}}+M
\end{align*}
Where by \eqref{eq:data}, \eqref{est:nl1} and \eqref{M} $C(M)$ is locally bounded and increasing in $M$.
    \begin{proposition}
      \label{regul}
      $\Gamma$ from \eqref{IVPe:Duh} is a contraction map on $X^M_{I_\ep}$ with $I_\ep = [0,T_\ep]$ and hence \eqref{IVPe} has a unique solution $\vue$ in $X^M_{I_\ep}$.
       \end{proposition}
     We note, that \eqref{IVPe:Duh} is locally well-posed in $C^0_tH^4$ using the same proof. However, for the arguments in section \ref{secuni} and beyond we need $\vue$ to be in $X^M_{I_\ep}$ .
      \begin{proof}
Let $\vu \in X^M_{I_\ep}$. Then by Minkowski inequality and \eqref{est:par}, followed by \eqref{eq:data}, \eqref{est:nl} and \eqref{M} we get:
        \begin{align*}
          \norm{\Gamma \vu(t)}_{C^0_{[0,T_\ep]}H^{8,2}} 
           \le \norm{\vu_0}_{H^{8,2}} + \norm{\vf}_{L^1_{[0,T_\ep]}H^{8,2}}  + (T_\ep+\sqrt[4]{\frac{T_\ep}{\ep^3}}) \norm{N\vu}_{C^0_{[0,t]}H^{5,2}}\\
            \le \frac{M}{2} + (T_\ep+\sqrt[4]{\frac{T_\ep}{\ep^3}}) C_{11,M}(1+M^{11}) M \le M
        \end{align*}
       Differentiating \eqref{IVPe:Duh} in $t$ and using $\norm{\Gamma \vu(t)}_{C^0_{[0,T_\ep]}H^{8,2}}\le M$ shows that $\norm{\dt\Gamma\vu}_{C^0_I H^{4,2}} \le C(M)$ and hence $\Gamma(X^M_{I_\ep}) \subset X^M_{I_\ep}$.\\

       Likewise using Lemmas \ref{lem:par} and \ref{lem:Mos} for the difference gives the contraction property.
      \end{proof}
      \begin{corollary}\label{cont}
        As we proved Proposition \ref{regul} by the contraction mapping argument, we automatically get continuous (and Lipschitz) dependence on data. That is, the flow map $(\vu_0,\vf) \to \vu$ is continuous in $Y \to X_{I_\ep}^M$.\\
        Moreover, \eqref{IVPe:Duh} has a persistence of regularity property.
      \end{corollary}

For the persistence of regularity, let the data $(\vu_0,\vf)$ in addition to satisfying \eqref{eq:data} satisfy for some $s>8$
\begin{align*}
  \norm{\vu_0}_{H^{s,2}}+\norm{\vf}_{L^1_{[-1,1]}H^{s,2}}+\norm{\vf}_{C^0_{[-1,1]}H^{s-4,2}} < \infty
\end{align*}
 then the solution of \eqref{IVPe:Duh} $\vue \in X^M_{T'}$ satisfies $\norm{\vue}_{C^0_{[0,T']}H^{s,2}\cap C^1_{[0,T']} H^{s-4,2}}  <\infty$. To see this, redo the boundedness argument in the proof of the Proposition \ref{regul} for $H^{9,2}$ partitioning $[0,T]$ into identical intervals of length $\tld T_9= \tld T(\ep,M,C_{12,M})$ for using \eqref{est:nl1}. For higher norms proceed inductively redoing \eqref{est:nl1} to estimate $\norm{\vue}_{C^0_{[0,T]}H^{s,2}}$ using intervals of length $\tld T_s =\tld T(\ep,\tld M_s, C_{s+3,\tld M_s})$ for $\tld M_s = \norm{\vue}_{C^0_{[0,T]}H^{s-1,2}}$.
  \renewcommand{\theenumi}{(\arabic{enumi})}
\renewcommand{\labelenumi}{\theenumi}
 \subsection{Adapting a linear estimate}\label{secuni}
 We now aim to adapt the Theorem \ref{lem:apri}, in order to estimate $\vu^\ep$ uniformly in $\ep$. As the Theorem \ref{lem:apri} applies to linear equations, we fix the coefficients \eqref{eq} at an arbitrary $\vv \in X^M_{T'}$ to get
\begin{align*}
\begin{cases}
  \dt \vu + N_{\vv} \vu = -\ep \dx^4 \vu + \vf\\
  \vu(x,0) = \vu_0(x)
\end{cases}
\end{align*}
with $N_{\vv}$ from \eqref{Nu}. We will show in Lemma \ref{coeff} that the coefficient bounds (L1)-(L3) for this linear equation do not depend on the particular $\vv$, but only on the (NL1)-(NL3) bounds, data for $t=0$ and the bounds $M$, $C(M)$ for $t\neq 0$.\\

In particular an application of Theorem \ref{lem:apri} for $\vu=\vv=\vue$ would imply for $T'\le T$
\begin{align}\label{unif:L2}
  \norm{\vue}_{\linf_I L^2_x} \le A(\norm{\vu_0}_{L^2} + \norm{\vf}_{L^1_I L^2_x})
\end{align}

We then pursue the same strategy for $\dx^{s}\vue$ and $\jap{x}^2 \dx^{s-6}\vue$ for $s=6,\ldots 14$. Namely, we differentiate \eqref{eq} $s$ times and account for quasi-linear interactions.
 \begin{equation}\label{eq:Hs}
 \left\{
   \begin{array}{c}
     \dt \dx^{s}\vu^{\ep} + a\left(x,t,(\dx^\alpha \vue)_{\alpha\le 2}\right)\cdot\dx^{{s}+3}\vu^{\ep}  + {b}^s\left(x,t,(\dx^\alpha \vue)_{\alpha\le 3}\right)\cdot\dx^{{s}+2}\vu^{\ep}\\ + {c}^s\left(x,t,(\dx^\alpha \vue)_{\alpha\le 4}\right)\cdot\dx^{{s}+1}\vu^{\ep} + {d}^s\left(x,t,(\dx^\alpha \vue)_{\alpha\le 5}\right) \dx^{s}\vu^{\ep}= - \ep\dx^{{s}+4}\vu^{\ep}\\
      +\dx^{s}\vf(x,t) +\vec{F}^s\left(x,t,(\dx^\alpha \vue)_{\alpha\le s-1}\right) \text{ on } \R\times I\\
     \vu^{\ep}(x,0) = \vu_0(x) \text{ on } \R\times\{0\}
   \end{array}
 \right.
  \end{equation}
  with $a$ the same as in \eqref{eq},
  \begin{align*}
     {b}^s\left(x,t,(\dx^\alpha \vue)_{\alpha\le 3}\right) = s \dx[ a(x,t,(\dx^\alpha \vue)_{\alpha \le 2} ] + \partial_{\vz^2}a(x,t,(\dx^\alpha \vue)_{\alpha \le 2}) \cdot\dx^3 \vu^\ep\\
      + b(x,t,(\dx^\alpha \vue)_{\alpha \le 2}) + \partial_{\vz^2} b (x,t,(\dx^\alpha \vue)_{\alpha \le 2})\cdot\dx^2 \vu^\ep
  \end{align*}
where $\vz=(\vz^\alpha)_{\alpha \le 2} = (\dx^\alpha \vue)_{\alpha\le 2}$. Similarly, $ c^s$ explicitly depend on a number $s$, $(\dx^\alpha \vue)_{\alpha\le 4}$ and up to $2$ derivatives of $a$, $b$; and $d^s$ depend on $s$, $(\dx^\alpha \vue)_{\alpha\le 5}$ and up to $3$ derivatives of $a$-$c$ respectively. The reason for these formulas, is that differentiation is "linearizing", so we only can get $a\dx^{s+3}\vue$ by applying all derivatives on $\vue$ and there are only few nonlinear ways to get terms higher than order $s$.\\

Once we establish (L1)-(L3) coefficient estimates for \eqref{eq:Hs} that are uniform when evaluated at $\vv \in X^M_{T'}$ in Lemma \ref{coeff}, then for the solution $\vue$ satisfying
\begin{equation}\label{reg:apr}
  \vue \in C^0_{I}H^{8+5,2}\cap C^1_I H^{4+5,2}
\end{equation}
We would then apply Theorem \ref{lem:apri} to \eqref{eq:Hs} to get
\begin{equation}\label{unif:Hs}
  \begin{split}
    \norm{\dx^{s}\vu^\ep}_{C^0_{I'} L^2} \le A\left(\norm{\dx^{s}\vu_0}_{L^2} + \norm{\dx^{s}\vf}_{L^1_{I'} L^2}\right) + A T'\norm{\vec F^s}_{\linf_{I'} L^2}
 \end{split}
\end{equation}
Provided $T'\le T$. Note that, \eqref{reg:apr} is needed for to ensure that \eqref{eq:Hs} is valid classically for $6\le s \le 14$.\\

Finally, we multiply \eqref{eq:Hs} by $\jap{x}^2$ to rewrite it as
\begin{align}\label{eq:wt}
\begin{cases}
    & \dt( \jap{x}^2\dx^{s}\vu^{\ep}) + a\dx^3(\jap{x}^2\dx^{s}\vu^{\ep}) + {b}^s\dx^2(\jap{x}^2\dx^{s}\vu^{\ep}) + {c}^s\dx(\jap{x}^2\dx^{s}\vu^{\ep})\\
    & + {d}^s\jap{x}^2\dx^{s}\vu^{\ep} = -\ep\dx^4(\jap{x}^2\dx^{s}\vu^{\ep})+\jap{x}^2\vec{F}^s(x,t,\vu^{\ep},\ldots,\dx^{{s}-1}\vu^{\ep})\\
    & +\jap{x}^2\dx^{s}\vf(x,t) + E_3^s\jap{x}\dx^s \vu^\ep\\
    & \jap{x}^2 \dx^s\vu^{\ep}(x,0) = \jap{x}^2\dx^{s}\vu_0(x)
\end{cases}
\end{align}
where $a$, $ b^s$, $ c^s$, $ d^s$, $\vec F^s$ are identical to \eqref{eq:Hs} by construction and $E_3 = [\ep\dx^4 + a\dx^3 + b^s \dx^2 + c^s\dx, \jap{x}^2]\jap{x}^{-1}$ is an order $3$ differential operator with $\linf$ coefficients. Hence as long as the coefficient estimates are established and \eqref{reg:apr} is valid, Theorem \ref{lem:apri} implies
\begin{align}\label{unif:wt}
\begin{split}
    & \norm{\jap{x}^2\dx^{s}\vu^{\ep}}_{C^0_{I'} L^2} \le A\left(\norm{\jap{x}^2\dx^{s}\vu_0}_{L^2} + \norm{\jap{x}^2\dx^{s}\vf}_{L^1_{I'} L^2}\right)\\
    & + A T'\left(\norm{\jap{x}^2\vec{F}^s}_{L^1_{I'} L^2}+\norm{E_3\jap{x}\dx^s \vu^\ep}_{L^1_{I'} L^2}\right)
\end{split}
\end{align}
 provided $T'\le T$.\\

Therefore, establishing a uniform estimate for $\norm{\vue}_{X^M_I}$ reduces to the following:
\begin{itemize}
  \item Find uniform bounds on the coefficients to \eqref{eq} and \eqref{eq:Hs} to use the Theorem \ref{lem:apri} (note the remark after Theorem \ref{lem:apri}).
  \item Control the terms not involving data ($\vec F^s$ and $E_3^s$) by choosing $T$ small enough.
  \item Remove the extra assumption \eqref{reg:apr}.
\end{itemize}
\subsection{Uniform estimate}\label{coeff:sec}
\begin{lemma}\label{coeff}
  Let $\vu \in X^M_{T'}$ and consider the equation \eqref{eq:Hs} for $0\le s \le 14$ with coefficients evaluated at $\vu$, i.e.
  \begin{align*}
    \dt \dx^{s}\vu^{\ep} + a\left(x,t,(\dx^\alpha \vu)_{\alpha\le 2}\right)\cdot\dx^{{s}+3}\vu^{\ep}  + {b}^s\left(x,t,(\dx^\alpha \vu)_{\alpha\le 3}\right)\cdot\dx^{{s}+2}\vu^{\ep}\\ + {c}^s\left(x,t,(\dx^\alpha \vu)_{\alpha\le 4}\right)\cdot\dx^{{s}+1}\vu^{\ep} + {d}^s\left(x,t,(\dx^\alpha \vu)_{\alpha\le 5}\right) \dx^{s}\vu^{\ep}= - \ep\dx^{{s}+4}\vu^{\ep} + \tld f
  \end{align*}
  where $ b^0 = b$, etc. Then the coefficients of this equation satisfy assumptions (L1)--(L3) from the Theorem \ref{lem:apri} with parameters $\lambda = \lambda(R)$, $ \tld C_0 = \tld C_0 (s,R, C_{3,R}, C_0)$ and $C_1 = C_1(s,M, C_{3,M}, C_0, C(M))$ with $\lambda(M)$, $C_{J,M}$ and $C_0$ from (NL1)-(NL3), and C(M) from \eqref{XT}.
\end{lemma}
  \begin{proof}
 To simplify notation, let $\vv=(\vu,\dx\vu,\dx^2\vu)$. Then $\norm{\vv}_{C^0_{I'}H^{8,2}} \approx_{8,2} \norm{\vu}_{C^0_{I'}H^{8,2}}$ and likewise for $\vv_0 = \vv(x,0)$ and $\dt \vv$. Note that \eqref{eq:data} becomes $\norm{\vv_0}_{H^{8,2}} \les R$.\\

By (NL1) $a_{ij}(x,t,\vv)=a_{ji}(x,t\vv)$ and $a(x,0,\vv_0) \ge \lambda(R)$ as by the Sobolev embedding $\norm{\vv_0(x)}_{\linf_x} \les \norm{\vu_0(x)}_{H^3} \le R$. Hence $a(x,t,\vv)$ satisfies (L1). \\

To verify (L2) we use (NL2), Sobolev embedding and $\norm{\vv_0}_{H^{8,2}} \les R$:
\begin{align*}
 \sum_{j=0}^3 \norm{\dx^j \left[a(x,0,\vv_0(x)\right]}_{\linf} \le   \sum_{\gamma_0 +\cdots + \gamma_{\abs{\beta}}=0}^3 \norm{\dx^{\gamma_0}\dz^\beta a(\vv_0)\dx^{\gamma_1}\vv_0\cdots\dx^{\gamma_{\abs{\beta}}}\vv_0}_{\linf}\\
  \le C_{3,R}(1+ R^3)
\end{align*}
Proceeding analogously and using \eqref{XT}, we get
\begin{align*}
  \sum_{j=0}^3 \norm{\dt \dx^j a(x,t,\vv)}_{\linf} \le C_{3,M}(1+ M^3) (1+C(M))
\end{align*}
and likewise for $\tld b^s$ --$\tld d^s$.\\

For (L3), we compute
\begin{align*}
  & \norm{\jap{x}^2\dx \left[a(x,0,\vv_0)\right]}_{\linf} \le \norm{\jap{x}^2\left(\dx a(x,0,\vv_0)+\dz a(x,0,\vv_0)\cdot\dx \vv_{0}\right)}_{\linf}
\end{align*}
Using the Fundamental Theorem of Calculus
\[
\dx a(x,0,\vv_0) = a(x,0,\vo)+ \int_0^1 \dx\dz a(x,t,\theta\vv_0) d\theta\cdot \vv_0
\]
we get using (NL2), (NL3) and Sobolev embedding
\begin{align*}
  \norm{\jap{x}^2\dx \left[a(x,0,\vv_0)\right]}_{\linf} \le  \norm{\jap{x}^2\dx a(x,0,\vo)}_{\linf}\\
   + \sup_{\abs{\vz}\le \norm{\vv_0}_{\linf}; a_i\le 1}\norm{\dx^{\alpha_1}\dz^{\alpha_2} a(x,0,\vz)}_{\linf}\norm{\jap{x}^2\dx^{\alpha_3}\vv_0}_{\linf} \le C_0 + C_{2,R}R
\end{align*}
 Similarly,
\begin{align*}
  \norm{\jap{x}^2\dt \dx [a(x,t,\vv(x,t)]}_{\linf_{I',x}} \le C_0 + 2 C_{2,M}(1+M)(1+C(M))
\end{align*}
Verifying (L3) for $b^s$--$d^s$ is similar and involves Taylor expansion of coefficients at $\vz=0$, chain rule and Sobolev embedding.
  \end{proof}
  \begin{lemma}\label{remai}
    Let $\vec F^s$ be from \eqref{eq:Hs} and $E_3^s$ from \eqref{eq:wt} for $\vue \in X^M_{T'}$. Then for $0\le s\le 8$
    \begin{align*}
      \norm{\vec F^{s+6}(x,t,(\dx^\alpha\vue)_{\alpha <s+6})}_{\linf_IL^2} \le C(s+6) C_{s+6,M} (1+M^{s+6})\norm{\vue}_{\linf_I H^{s+6}} \\
\norm{\jap{x}^2\vec F^{s}(x,t,(\dx^\alpha\vue)_{\alpha <s})}_{\linf_IL^2} \le C(s) C_{s,M} (1+M^{s})\norm{\jap{x}^2\vue}_{\linf_I H^s}\\
\norm{E_3^s\jap{x}\dx^s\vue}_{\linf_I L^2} \le C(s) C_{s,M}(1+M^4)(\norm{\jap{x}^2\dx^s \vue}_{\linf_I L^2}+ \norm{\dx^s\vue}_{\linf_I H^6})
    \end{align*}
  \end{lemma}
  \begin{proof}
  The first two estimates follow immediately from the construction of $\vec F^s$ and Sobolev embedding.\\

  For the second, by construction $E_3$ is a differential operator with coefficients bounded by $C(s)C_{s,M}(1+M^4)$ and hence
  \begin{align*}
    \norm{E_3^s\jap{x}\dx^s\vue}_{\linf_I L^2} \le C(s) C_{s,M}(1+M^4)\norm{\jap{x}\dx^s\vue}_{\linf_I H^3}\\
     \le C(s)C_{s,M}(1+M^4)\left( \norm{\jap{x}^2\dx^s \vue}_{\linf_I L^2}+ \norm{\dx^s\vue}_{\linf_I H^6}\right)
  \end{align*}
  Where the last line follows by interpolation in \eqref{wt:int}.
  \end{proof}
\begin{theorem}
   \label{thm:unif}
There exists a constant $A$ obtained using the Theorem \ref{lem:apri} and depending on data, more precisely $A=A(R,\lambda_R,C_0, C_{14,R})$, and\\ $T=T(M,\lambda_R,C_0, C_{14,M},C(M))$ such that for any $I\subset [0,T]$ a solution $\vue \in X^M_I$ of \eqref{IVPe} on $I$ satisfies
 \begin{equation}\label{eq:unif}
          \norm{\vue}_{C^0_{I}H^{8,2}} \le 8A R
        \end{equation}
        Moreover, we can extend $\vue$ to be in $X^M_{[0,T]}$ and solve \eqref{IVPe} on $[0,T]$.
\end{theorem}
\begin{proof}
Note, that the Proposition \ref{regul} holds on $[T',T'+T_\ep]$ with data $(\vu^\ep(T'),\vf)$ at $t=T'$, as long as $\norm{\vu^\ep(T'),\vf}_Y < \frac{M}{2}$, which will hold for any $T'\in I$ as long as \eqref{eq:unif} is valid.
 When this is the case, $\vue$ can be extended to solve \eqref{IVPe} on $[0,T'+T_\ep]$ in $X^M_{[0,T' + T_\ep]}$ and arguing inductively to extend $\vue$ to $I=[0,T]$. Thus it suffices to prove  \eqref{eq:unif}.\\

We first complete the proof under the assumption that \eqref{reg:apr} is valid.\\

 By the Lemma \ref{coeff} estimates \eqref{unif:L2}, \eqref{unif:Hs} and \eqref{unif:wt} are valid. Adding these inequalities and using equivalence of norms \eqref{def:norms} we get
 \begin{align*}
   \norm{\vue}_{C^0_I H^{8,2}} \le A(\norm{\vu_0}_{H^{8,2}} + \norm{\vf}_{L^1_I H^{8,2}}) + AT\norm{\vec F^{14}}_{\linf_IL^2}\\
    + AT\left(\sum_{s=1}^8\norm{\jap{x}^2\vec F^{s}}_{\linf_IL^2}+ \sum_{s=0}^8\norm{E_3^s\jap{x}\dx^s\vue}_{\linf_I L^2}\right)
 \end{align*}
 Using Lemma \ref{remai} we get
 \begin{align*}
   \norm{\vue}_{C^0_I H^{8,2}} \le A(\norm{\vu_0}_{H^{8,2}} + \norm{\vf}_{L^1_I H^{8,2}}) + AT C_{14,M}(1+M^{18})\norm{\vue}_{C^0_I H^{8,2}}
 \end{align*}
 Which after choosing $T$ small enough proves \eqref{eq:unif}.\\

 To remove assumption \eqref{reg:apr} we regularize the data and note that \eqref{eq:unif} is uniform for every $\norm{(\vu_0,\vf)}_Y <R$. More precisely, we convolve $(\vu_0,\vf)$ with a mollifier $\chi_m(x) = m\chi(m x)$, where $\chi \in C^\infty_0(\R)$ and $\int \chi(x) dx = 1$ to get $(\vu_{0,m},\vf_m)\in \Schw \times C^0_{[0,1]} \Schw$ and $\lim_{m\to \infty} (\vu_{0,m},\vf_m) = (\vu_0,\vf)$ in $Y$. Therefore, for $m$ large enough $\norm{(\vu_{0,m},\vf_m)}_Y <R$ and by the Corollary \ref{cont} the corresponding solution $\vu^\ep_m \in X^M_{I_\ep}\cap C^0_I H^{8+5,2}\cap C^1_I H^{4+5,2}$ and converges to $\vu^\ep$ in $X^M_I$ whenever both are defined.\\

 Applying Theorem \ref{thm:unif} to $\vue_m$, for which \eqref{reg:apr} applies, allows us to recover \eqref{eq:unif} for $\vu$ in the limit as $m\to \infty$,  as the constants $A$ and $T$ don't depend on $m$.
\end{proof}

\section{Removing regularization}\label{seclim}
We first construct solution $\vu$ of \eqref{eq} in a topology weaker than $C^0_I H^{8,2}\cap C^1_I H^{4,2}$.
\begin{proposition}
  \label{conv}
 There exists a $T>0$ and $I=[0,T]$ and a sequence of solutions of \eqref{IVPe} $\vu_{\ep_n}$ in $X^M_I$ for $\ep_n\to 0$, such that $\vu_{\ep_n} \to \vu$ in $C^0 H^{7,2}\cap C^1_I H^{4,2}$.
\end{proposition}
\begin{proof}
Take a difference of solutions $\vue$ and $\vuee \in X^M_I$ of \eqref{IVPe} for $0<\ep'<\ep \le 1$:
\begin{align}\label{diff}
\begin{cases}
  \dt (\vue-\vuee) + N_{\vue}\vue - N_{\vuee}\vuee = -\ep \dx^4(\vue-\vuee) + (\ep-\ep')\dx^4\vuee\\
  \vue-\vuee(x,0)\equiv 0
\end{cases}
\end{align}
Then we can write $N_{\vu^\ep}\vu^\ep - N_{\vu^{\ep'}}\vu^{\ep'}=L_{\vue,\vuee}(\vue-\vuee)$ where by the Fundamental Theorem of Calculus $L_{\vue,\vuee}$ has the form
\begin{align*}
\begin{split}
  &     {L}_{\vue,\vuee} =a(x,t,(\dx^j \vu^{\ep})_{j\le 2})\dx^3
     + \tld{b}(x,t,(\dx^j \vu^{\ep})_{j\le 3},(\dx^j \vu^{\ep'})_{j\le 3})\dx^2\\
     & + \tld{c}(x,t,(\dx^j \vu^{\ep})_{j\le 3},(\dx^j \vu^{\ep'})_{j\le 3})\dx
     + \tld{d}(x,t,(\dx^j \vu^{\ep})_{j\le 3},(\dx^j \vu^{\ep'})_{j\le 3})I
\end{split}
  \end{align*}
with coefficients
 \begin{align*}
& \tld{b}_{ij} = b_{ij}(x,t,\vu_\ep,\dx \vu_\ep,\dx^2 \vu_\ep)
 + \dx^3u_{\ep',k}\int_0^1 \partial_{z^2_j} a_{ik}(x,t,(\dx^\alpha\left((1-\theta)\vu_\ep + \theta\vu_{\ep'}\right)_{\alpha\le 2})) d\theta
\end{align*}
   and similarly for $\tld{c}$, $\tld{d}$.\\

By an argument identical to the Lemma \ref{coeff}, for $\vu$, $\vv \in X^M_I$, $L_{\vu,\vv}$ satisfies (L1)--(L3) with bounds $\tld C_0$, $C_1$ dependent on the same parameters as in Lemma \ref{coeff}. Hence for an appropriate $T$ Theorem \ref{lem:apri} applied to \eqref{diff} gives
\begin{align}\label{diff:L2}
  \norm{\vue-\vuee}_{C^0_IH^{8,2}} \le A\norm{(\ep-\ep')\dx^4\vuee}_{L^1_I L^2_x}\le ATM(\ep-\ep')
\end{align}
We then use that $\norm{\vue-\vuee}_{C^0_IH^{8,2}} \le 2M$ and for $0<1 \le 1$ interpolate $H^{13}$ between $L^2$ and $H^{8+6}\subset H^{8,2}$ to get
\begin{align}\label{diff:Hs}
\begin{split}
  & \norm{\vu^{\ep} - \vu^{\ep'}}_{\linf_tH^{13}_x} \le \norm{\vu^{\ep} - \vu^{\ep'}}_{L^{\infty}_I L^2_x}^{\frac{1}{14}}\norm{\vu^{\ep} - \vu^{\ep'}}_{\linf_tH^{20}_x}^{\frac{13}{14}} \le [(\ep - \ep')AT]^{\frac{1}{14}}{(2M)}^{\frac{13}{14}}
\end{split}
\end{align}
Hence by completeness of $C^0_I H^{13}_x$ we can take a sequence $\ep_n\to 0$ to construct $\vu = \lim_{\ep_n \to 0} \vue$.\\

  Note, that $\linf_IH^{14}= \left(L^1_IH^{-14}\right)^{*}$, and hence the closed unit ball in it is weak-$*$ compact by the Banach-–Alaoglu theorem. As a consequence, the sequence $\vu^{\ep_n} \in X^M_I \subset  \linf_I H^{8,2}_x$ has a weak limit $\vu^{\ep_n} \weakto \vv \in  \linf_IH^{14}_x$ up to a subsequence. Hence, for a.e. $x$, there is a further subsequence, such that $\vu^{\ep_n}(x) \to \vv(x)$. But by uniqueness of limits $\vv = \vu$. A similar argument for $\linf_I H^8(\jap{x}^4 dx)$ gives
\begin{equation}\label{eq:limit:weak}
  \vu^{\ep_n} \weakto \vu \in \linf_I H^{8,2}
\end{equation}

Using Fatou lemma gives $\norm{\vu}_{\linf_IH^{8,2}_x}\le {M}$.\\

For strong convergence in $C^0_I H^{8-1,2}$ we multiply \eqref{diff} by $\jap{x}^2$ and rewrite it as
\begin{align*}
  \begin{cases}
  \dt \jap{x}^2(\vue-\vuee) + L_{\vue,\vuee}\jap{x}^2(\vue-\vuee) = -\ep \dx^4\jap{x}^2(\vue-\vuee)\\ + (\ep-\ep')\jap{x}^2\dx^4\vuee+ E_3\jap{x}(\vue-\vuee)\\
  \jap{x}^2(\vue-\vuee)(x,0)\equiv 0
\end{cases}
\end{align*}
with $E_3 = [-\ep\dx^4-L_{\vue,\vuee},\jap{x}^2]$ is an order $3$ differential operator with coefficients bounded by $4+C_{1,M}(1+M)$. Applying Theorem \ref{lem:apri} all terms are treated as in the proof of \eqref{diff:Hs}, except we use \eqref{wt:int} to interpolate
\begin{align*}
  \norm{E_3\jap{x}(\vue-\vuee)}_{\linf_I L^2_x} \les_{M,C_{1,M}}\norm{\vue-\vuee}_{\linf_I H^6}^{\frac{1}{2}}\norm{\jap{x}^2(\vue-\vuee)}_{\linf_I L^2_x} \to 0
\end{align*}
as $\ep$, $\ep \to 0$. Interpolating $H^{8-1}$ between $L^2$ and $H^8$ implies $\jap{x}^2\vu_{\ep_n} \to \jap{x}^2\vu$ in $C^0_I H^{8-1}$.\\

Finally, using \eqref{diff} and Lemma \ref{lem:Mos} implies $\dt \vu_{\ep_n} \to \dt \vu \in C^0_I H^{4,2}$.
\end{proof}

Moreover, for uniqueness of the limit we use \eqref{diff:Hs} for $\vue=\vu$ and $\vuee=\vu'$ for $\ep=\ep'=0$. The regularity in the Proposition \ref{conv} is enough for the Lemma \ref{coeff} to be valid. Moreover, as Theorem \ref{lem:apri} is valid for $[-T,T]$, when $\ep = 0$ we get uniqueness for $[-T,T]$.\\

We now aim to recover the loss of the derivative in Proposition \ref{conv} following the regularization method of Bona-Smith \cite{BS75}. That is we regularize the data with a parameter $\kappa$ and then apply the Theorem \ref{lem:apri} for the difference at the top level of regularity keeping a careful track of $\kappa$.

\subsection{Data regularization}\label{sec:Bona}
Let $0\le \phi(\xi) \le 1$ be a radial smooth bump
     \[
     \phi(\abs{\xi})=      \begin{cases}
      & 1,\text{ if }\abs{\xi} \le 1\\
       & 0,\text{ if }\abs{\xi} \ge 2
     \end{cases}
     \]
     Define for all $0<\kappa\le 1$, $\hat u_{0,\kappa} = \hat u_0\cdot \phi(\kappa\abs{\xi})$.
 \begin{lemma}
   \label{BS} Let $\K \Subset H^{14}$ be a compact set. Then $\forall \kappa>0$, and any $\vu_0 \in \K$, $\vu_{0,\kappa} \in \Schw$ satisfies
   \begin{subequations}
     \begin{align}
      \norm{\vu_{0,\kappa}}_{H^{14+j}} = O(\kappa^{-j})       \text{ for all } j\ge 0\,\,\,  \label{eq:BS:above}\\
     \norm{\vu_{0,\kappa}-\vu_0}_{L^2} = o(\kappa^{14}) \text{ and } \norm{\vu_{0,\kappa}-\vu_0}_{H^{14}} = o(1)\label{eq:BS:L2}
   \end{align}
   \end{subequations}
   with the convergence rate dependent on $\K$.
 \end{lemma}
\begin{proof}
     Let $j\ge 0$ and $0<\kappa <1$
     \begin{align}\label{BSP1}
     \begin{split}
       & \norm{u_{0,\kappa}}_{H^{14+j}}^2  = \int_{\abs{\xi} \le \frac{2}{\kappa}} (1+\abs{\xi}^2)^{14}\abs{\hat u_0(\xi) }^2 (1+\abs{\xi}^2)^j \phi(\kappa\abs{\xi})^2 d\xi \le  (\frac{3}{\kappa})^{2j} \norm{u_0}_{H^{14}}^2
     \end{split}
     \end{align}
     Which proves \eqref{eq:BS:above}.\\
      For \eqref{eq:BS:L2} is suffices to show the first estimate, with the second done identically. By the Fundamental Theorem of Calculus:
     \begin{align*}
      & \norm{u_{0,\kappa}-u_0}_{L^2}^2 = \int (\phi(\kappa\abs{\xi})-1)^2\abs{\hat u_0(\xi)}^2 d\xi \le \int \left(\sup_{\eta \in B_{\kappa\abs{\xi}}(0)}\abs{\phi'(\eta)}\right)^2 \kappa^2\abs{\xi\hat u_0}^2 d\xi
     \end{align*}
     where we used $\phi(0)=1$ and defined $B_{\kappa\abs{\xi}}(0) = [-\kappa\abs{\xi},\kappa\abs{\xi}]$. As $\dxi^j\phi(0)=0$ for all $j>0$, we can continue with the Taylor expansion of $\abs{\phi^{(j)}(\eta)}$ $14$ times and then use $\dxi^j\phi\equiv 0$ on $B_1(0)$ for $j>0$ to conclude
     \begin{align*}
       & \norm{u_{0,\kappa}-u_0}_{L^2}^2  \le \kappa^{28} \norm{\dxi^{14} \phi}_{\linf}^2\int_{\abs{\xi}\ge \frac{1}{\kappa}} \abs{\xi^{14}\hat u_0}^2 d\xi
     \end{align*}
The $o(1)$ rate comes by the Lebesgue Dominated Convergence that is uniform for $\vu_0$ in a compact $K$.
\end{proof}

\begin{lemma}\label{BS:wt}
Lemma \ref{BS} is valid with weights. That is for $K\Subset  H^{8,2}$
     \begin{align*}
     \norm{\jap{x}^2\vu_{0,\kappa}}_{H^{8+j}} = O(\kappa^{-j})       \text{ for all } j\ge 0\\\
     \norm{\jap{x}^2(\vu_{0,\kappa}-\vu_0)}_{L^2} = o(\kappa^{8}) \text{ and } \norm{\jap{x}^2(\vu_{0,\kappa}-\vu_0)}_{H^{8}} = o(1)
   \end{align*}
\end{lemma}
\begin{proof}
The Lemma \ref{BS} is valid for $(\jap{x}^2\vu_0)_\kappa$ (with obvious modification of $14$ to $8$), so it suffices to show that $\vv_\kappa = \jap{x}^2\vu_{0,\kappa} - (\jap{x}^2\vu_0)_\kappa$ has norms $\norm{\vv_\kappa}_{H^s} = O(\kappa^{-N})$ for all $s$ and $N$.\\

Using $\jap{x}^2 - \jap{x-y}^2 = 2(x-y)\cdot y + \abs{y}^2$ we get
\begin{align*}
  \vv_\kappa(x) = \int \left(2(x-y)\cdot y + \abs{y}^2\right) \vu_0(x-y)\frac{1}{\kappa} \check{\phi} (\frac{y}{\kappa}) dy
\end{align*}
Hence
\begin{align*}
  \widehat{\vv_\kappa}(\xi) = \widehat{(2 x \vu_0)}(\xi)\kappa i\dxi \phi(\kappa\xi) - \widehat{\vu_0}(\xi)\kappa^2\dxi^2\phi(\kappa\xi)
\end{align*}
and as $\dxi^j\phi(0)=0$ for $j\ge 1$ we can get arbitrarily high power of $\kappa$ by Taylor expanding $\dxi\phi$. Thus remark follows, as $x\vu_0$ and $\vu_0 \in L^2$.
\end{proof}
\begin{remark}\label{rem:f}
Lemmata \ref{BS} and \ref{BS:wt} hold for $\K \Subset L^1_I H^{8,2}$ or $\K\Subset C^0_IH^{8,2}$. That is for each $f\in \K$, there is $f_\kappa \in L^1_I \Schw$, such that \begin{align*}
     & \norm{\vf_\kappa-\vf}_{L^1_{I}L^2} = o(\kappa^{14}) \,\,\,     \norm{\jap{x}^2(\vf_\kappa-\vf)}_{L^1_{I}L^2} = o(\kappa^{8})\\
      & \norm{\vf_{\kappa}-\vf}_{L^1_{I}H^{8,2}} = o(1) \,\,\, \norm{\vf_{\kappa}}_{L^1_{I}H^{8+j,2}} = O(\kappa^{-j})       \text{ for all } j\ge 0\,\,\,
   \end{align*}
   and similarly for $K\Subset C^0_I H^{4,2}$.
\end{remark}
For $L^1H^{8,2}$ redo the Lemmas using
\begin{align*}
\lim_{\kappa \to 0} \int_I(\int_{\abs{\xi}\ge \frac{1}{\kappa}} (1+\abs{\xi}^2)^{14}\abs{\hat \vf}^2 d\xi)^{\frac{1}{2}} dt = 0,
\end{align*}
for $o(1)$ convergence rate, which follows from Lebesgue Dominated convergence in $t$ and $\xi$.\\

Finally, for $f\in C^0_I H^{4,2}$ we use the continuity of $f$ to see that $f(I)$ is a compact set in $H^{4,2}$.\\

Putting together Lemmata \ref{BS} and \ref{BS:wt} and Remark \ref{rem:f} we get
\begin{proposition}\label{BS:data}
\begin{equation*}
  \norm{(\vu_0,\vf)-(\vu_{0,\kappa},\vf_\kappa)}_Y = o(1)
\end{equation*}
and $(\vu_{0,\kappa},\vf_\kappa)\in H^{8+s,2}\times (L^1_{[-1,1]}H^{8+s,2}\cap C^0_{[-1,1]}H^{4+s,2})$ with $O(\kappa^{-s})$ norms for $s\ge 0$.
\end{proposition}

We now want to improve the result of the Proposition \ref{conv} for the solution $\vu^\ep_\kappa$ of \eqref{IVPe} for data $(\vu_{0,\kappa},\vf_\kappa)$. By the Proposition \ref{BS:data}, for $\kappa>0$ small enough $\norm{(\vu_{0,\kappa},\vf_\kappa)}_Y <R$, hence the Theorem \ref{thm:unif} applies to the $\vue_\kappa$. Moreover, by Proposition \ref{regul}, $\vue_\kappa \in C^0_I H^{8+s,2}\cap C^1_I H^{4+s,2}$ for any $s\ge 0$. In fact, $\vue_\kappa$ satisfies the following uniform bound
\begin{proposition}\label{pers}
For interval $I$ from Theorem \ref{thm:unif} there exist a constant $C$ independent of $\ep$ and $\kappa$, such that for $j=1,2,3$:
\begin{align}\label{pers:eq}
  \norm{\vue_\kappa}_{C^0_I H^{8+j,2}} \le C (\norm{\vu_{0,\kappa}}_{H^{8+j,2}} +\norm{\vf_\kappa}_{L^1_I H^{8+j,2}}) =O(\kappa^{-j})
\end{align}
\end{proposition}
\begin{proof}
  Let $I_k = [k\tld T,(k+1)\tld T]$ for $\tld T$ to be fixed below and $0\le k\tld T \le T-\tld T$. Then considering \eqref{IVPe} with data $\left(\vue_\kappa(k\tld T),\vf(t+k\tld T)\right)$ Lemma \ref{coeff} is valid on $I_k$ and Lemma \ref{remai} can be refined to hold for $0\le s \le 8+j$. Namely, there's no way to get terms involving $\dx^{14}\vue\cdot\dx^{14+j}\vue$ by differentiating $N_{\vue}(\vue)$ $17$ times. This gives
  \begin{align*}
    \norm{\vue_\kappa}_{C^0_{I_k} H^{8+j,2}} \le C(\norm{\vue_\kappa(k\tld T)}_{H^{8+j,2}}
     + \norm{\vf_\kappa}_{L^1_I H^{8+j,2}})\\ + C\tld T C_{14+j,M}(1+M^{18+j})\norm{\vue}_{C^0_{I_k} H^{8+j,2}}
  \end{align*}
  with $C=A(\norm{(\vue_\kappa,\vf_\kappa)}_Y)$ from Theorem \ref{lem:apri}.
  Choosing $\tld T$ small enough and inducting on $k$ finishes the proof.
\end{proof}
Hence interpolating $\norm{\vue_\kappa-\vuee_\kappa}_{C^0_I H^{14}}$ between $L^2$ estimate in \eqref{diff:L2} and $H^{15}$ in \eqref{pers:eq} gives that $\vu_\kappa \in C^0_I H^{14}$. Doing the same with weights, implies
\begin{equation}\label{lim:reg}
\vu_\kappa \in C^0_I H^{8,2}
\end{equation}
\subsection{Uniform convergence}
\begin{theorem}\label{conv:BS}
  $\vu^\ep_\kappa \to \vu^\ep$ as $\kappa \to 0$ in $\linf_IH^{14}_x$ uniformly in $0<\ep\le 1$. Moreover, this convergence is uniform for $\vue$ coming from data in a compact set $K\Subset B_R(0)$, where $B_R(0)$ is a ball of radius $R$ centered at $0$ in $Y$.
\end{theorem}
\begin{proof}
  Redoing \eqref{diff:L2} for $\vue$ and $\vue_\kappa$ we get
   \begin{equation}\label{diff:BS}
   \norm{\vu^\ep-\vu^\ep_\kappa}_{\linf_IL^2_x}\le A (\norm{\vu_0-\vu_{0,\kappa}}_{L^2}+\norm{\vf-\vf_\kappa}_{L^1_IL^2_x})=o(\kappa^{14})
   \end{equation}
    where the last identity is by \eqref{eq:BS:L2} and Remark \ref{rem:f}.\\

We then subtract \eqref{eq:Hs} for $\vue_\kappa$ from \eqref{eq:Hs} for $\vue$ with $s=14$ and rewrite it as
  \begin{align*}
      \begin{cases}
    & \dt \dx^{14}(\vue-\vue_\kappa) + N_{14}(\vue-\vue_\kappa) =- \ep \dx^{18} (\vue-\vue_\kappa) +  \dx^{14}(\vf-\vf_\kappa)+\vec F^{14}-\vec F^{14}_\kappa\\
    & +(N_{14,\kappa}-N_{14})\vu^\ep_\kappa\\
    & \dx^{14}(\vue-\vue_\kappa)(x,0) = \dx^{14}(\vu_0(x)-\vu_{0,\kappa}(x))
  \end{cases}
   \end{align*}
   where $N_{14}$ is the operator from the Lemma \ref{coeff} for $s=14$ with coefficients evaluated at $\vue$ and $\vec F^{14} = \vec F^{14}(x,t,\vu^{\ep},\ldots,\dx^{13}\vu^{\ep})$, and likewise $N_{14,\kappa}$ and $F^{14}_\kappa$ are evaluated at $\vue_\kappa$ respectively. Therefore, assuming \eqref{reg:apr}, we apply the Theorem \ref{lem:apri}:
   \begin{align*}
    & \norm{\dx^{14}(\vue-\vue_\kappa)}_{L^{\infty}_I L^2_x} \le A(\norm{\dx^{14}(\vu_0-\vu_{0,\kappa})}_{L^2_x}
      + \norm{\dx^{14}(\vf-\vf_\kappa)}_{L^1_IL^2_x})\\
      & +A\norm{\vec F^{14}-\vec F^{14}_\kappa}_{L^1_IL^2_x}+A\norm{(N_{14,\kappa}-N_{14})\vu^\ep_\kappa}_{L^1_IL^2_x}\equiv I_1 + I_2 + I_3
  \end{align*}
 We estimate the right hand side term by term. By the Proposition \ref{BS:data}, $I_1 = o(1)$ as $\kappa \to 0$.\\
 For $I_2$ we write by the Fundamental Theorem of Calculus,
  \begin{align*}
    \vec F^{14}-\vec F^{14}_\kappa = \sum_{\alpha \le 13}\int_0^1 \der{F^{14}}{\vz^\alpha}\left(x,t,(\dx^\beta( \theta \vue + (1-\theta)\vue_\kappa))_{\beta\le 13}\right) d\theta \cdot \dx^{\alpha}(\vue-\vue_\kappa)
  \end{align*}
  Hence by Sobolev
\begin{align*}
  & I_2  \le AT C_{15,M} (1+M^{14})  \norm{ \vu^\ep - \vu^\ep_\kappa}_{\linf_I H^{13}_x}
\end{align*}
We then interpolate $H^{13}$ between $L^2$ and $H^{14}$  as in \eqref{diff:Hs} using $\eqref{diff:BS}$:
\begin{align*}
  & \norm{ \vu^\ep - \vu^\ep_\kappa}_{\linf_I H^{13}_x} \le \norm{ \vu^\ep - \vu^\ep_\kappa}_{\linf_I L^2_x}^{\frac{1}{14}} \norm{ \vu^\ep - \vu^\ep_\kappa}_{\linf_I H^{14}_x}^{\frac{13}{14}} = o(\kappa)
\end{align*}
  For $I_3$, we proceed as for $I_2$ using Fundamental Theorem of Calculus
  \begin{align*}
    (N_{14,\kappa}-N_{14})\vu^\ep_\kappa = \int_0^1 \der{a}{\vz^\alpha}\left(x,t,(\dx^\beta( \theta \vue_\kappa + (1-\theta)\vue))_{\beta\le 2}\right) d\theta \cdot \dx^{\alpha}(\vue_\kappa-\vue)\,\dx^{17}\vue_\kappa
  \end{align*}
  Hence by the Cauchy-Schwarz and Sobolev:
  \begin{align*}
      \norm{(N_{14,\kappa}-N_{14})\vu^\ep_\kappa}_{L^1_I L^2_x} \le C_{5,M}(1+M^{3})\norm{ \vu^\ep - \vu^\ep_\kappa}_{\linf_I H^{3}_x}      \norm{\dx^{17}\vu^\ep_\kappa}_{\linf_I L^2_x}
  \end{align*}
Thus by  \eqref{diff:BS} and \eqref{pers:eq}
  \begin{align*}
    I_3  =  o(\kappa^{17})O(\kappa^{-3}) = o(\kappa^{14})
  \end{align*}
  Combining the estimates for terms $I_1$-$I_3$, \eqref{diff:BS} and using equivalence of norms \eqref{def:norms} gives $\vue_\kappa \to \vue$ in $C^0_I H^{14}$ provided \eqref{reg:apr} is valid.\\

  To finish the proof without \eqref{reg:apr}, we proceed as in the proof \eqref{thm:unif}, except when we regularize the data $(\vu_{0,m},\vf_m) \to (\vu_0,\vf)$ in $Y$, we use compactness of $K=\{(\vu_{0,m},\vf_m)\}\cup \{(\vu_0,\vf)\}$ in $Y$ to ensure that the convergence rate $\vue_{m,\kappa} \to \vue_m$ as $\kappa \to 0$ is uniform in $m$.\\

  Finally, as the rate of convergence in $\kappa$ comes from Lemma \ref{BS} as longs as $\norm{(\vu_0,\vf)}_Y < R$, it remains valid for data in a compact set $K\Subset B_R(0)$.
  \end{proof}

Redoing the Proposition \ref{conv:BS} with weights, we get $\vue_\kappa \to \vue$ in $C^0_I H^{8,2}$.
\subsection{Time continuity and continuous dependence}
By the Proposition \ref{conv} $\vue-\vue_\kappa \weakto \vu-\vu_\kappa$ in $\linf_I H^{14}$. Thus for fixed $\kappa$, we apply Fatou's lemma to $\int \dx^{14}(\vue - \vue_\kappa)\cdot \vec \phi\, dx $ for a vector valued test function $\vec \phi \in C_0^\infty(\R)$ to get
  \begin{align}\label{conv:d}
  \norm{\vu-\vu_\kappa}_{\linf_I H^{14}_x} \le \liminf_{\ep \to 0}\norm{\vu^\ep-\vu^\ep_\kappa}_{\linf_I  H^{14}_x} = o(1)
  \end{align}
  as $\kappa \to 0$. Thus the Theorem \ref{conv:BS} is also valid with $\ep = 0$.\\

 Combining \eqref{conv:d} with \eqref{lim:reg}, gives $\vu \in C^0_I H^{14}$. Moreover,
  \begin{align*}
    \vue - \vu = (\vue-\vue_\kappa) + (\vu_\kappa - \vu)+(\vue_\kappa - \vu_\kappa)
  \end{align*}
  with the first two terms that go to $0$ as $\kappa \to 0$ in $C^0_I H^{14}$ uniformly in $\ep$ and the last term goes  $0$  for each fixed $\kappa$. Hence $\vue \to \vu$ in $C^0_I H^{14}$.\\

  Redoing the same argument with weights proves     $\vu \in C^0_I H^{8,2}$
  \begin{equation}\label{conv:top}
 \lim_{\ep \to 0} \norm{\vu^{\ep} - \vu}_{C^0_I H^{8,2}} = 0
  \end{equation}
    with the convergence rate dependent on the data profile, namely the compact set $\K\Subset Y$.
    \begin{proposition}\label{cont:prop}
        Suppose $(\vu_{0,m},\vf_m) \to (\vu_{0},\vf)$ in $Y$. Then the respective solutions of \eqref{eq} $\vu_m\to \vu$ in $X_I = C^0_IH^{8,2}\cap C^1_I H^{4,2}$.
    \end{proposition}
    \begin{proof}
      For $m$ large enough the data have size less that $R$. For such $m$ we write
      \begin{align*}
        \vu_m - \vu = (\vu_m - \vue_m) + (\vue - \vu) + (\vue_m - \vue)
      \end{align*}
      The first two terms converge to $0$ in $C^0_I H^{8,2}$ as $\ep \to 0$ uniformly in $m$ by compactness of data, while the third goes to $0$ as $m \to \infty$ for each fixed $\ep$ by the Corollary \ref{cont}.\\

      For the $C^1_IH^{4,2}$ use \eqref{eq} and Lemma \ref{lem:Mos}.
    \end{proof}
    Finally, for Persistence of Regularity for \eqref{eq}, proceed inductively $s=15, 16,\ldots$ as in the Proposition \ref{pers}, applying Theorem \ref{lem:apri} for \eqref{eq:Hs}, so that for each step terms from the Lemma \ref{remai} are controlled by the previous step. Then remove \ref{reg:apr} by the Proposition \ref{cont:prop}. This completes the proof of the Main Theorem \ref{thm:nl}.
\section*{Acknowledgement}
Results of this paper were obtained during my Ph.D. studies at the University
of Chicago and are also contained in my thesis \cite{Akhun}. I would like to express immense
gratitude to my supervisor Carlos Kenig, whose
guidance and support were essential for the successful completion of this project.\\

I would also like to thank A. Allan, F. Chung, N. Longo and W. Schlag for useful conversations regarding this work and the referee for offering helpful suggestions that improved the presentation of this work.

\medskip
Received xxxx 20xx; revised xxxx 20xx.
\medskip

\end{document}